\documentclass[preprint]{imsart}
\RequirePackage{amsthm,amsmath,amsfonts,amssymb}
\RequirePackage[colorlinks,citecolor=blue,urlcolor=blue]{hyperref}
\usepackage{graphics, graphicx, xcolor}

\startlocaldefs
\usepackage[authoryear]{natbib}
\def\mC{\mathbb{C}}
\def\mR{\mathbb{R}}
\def\tr{\mbox{tr}}
\def\cov{\mbox{cov}}
\def\var{\mbox{var}}
\def\sign{\mbox{sign}}

\def\diag{\mbox{diag}}
\def\corr{\mbox{corr}}

\newcommand{\bSig}{\mbox{\boldmath $\Sigma$}}

\newcommand{\bI}{\mathbf I}
\newcommand{\B}{\mathbf B}
\newcommand{\K}{\mathbf{K}}
\newcommand{\bR}{\mathbf{R}}
\newcommand{\h}{\mathbf{h}}
\newcommand{\bx}{\mathbf{x}}
\newcommand{\bz}{\mathbf{z}}
\newcommand{\M}{\mathbf M}

\newcommand{\A}{\mathbf{A}}
\newcommand{\Z}{\mathbf{Z}}
\newcommand{\bS}{\mathbf{S}}
\newcommand{\bH}{\mathbf{H}}
\newcommand{\W}{\mathbf W}
\newcommand{\bu}{\mathbf u}
\newcommand{\bv}{\mathbf v}
\newcommand{\X}{\mathbf{X}}
\newcommand{\bT}{\mathbf{T}}
\newcommand{\Q}{\mathbf Q}
\newcommand{\E}{\mathbb E}
\newcommand{\ovec}{\mathbf{0}}

\newcommand{\di}{\mbox{diag}}
\newcommand{\trans}{^{ \mathrm{\scriptscriptstyle T}}}

\def\defby{\stackrel{\mbox{\textrm{\tiny def}}}{=}}

\newcommand\lb{\left(}
\newcommand\rb{\right)}

\def\indist{\,{\buildrel d \over= }\,}
\numberwithin{equation}{section}
\numberwithin{table}{section}

\newtheorem{theorem}{Theorem}[section]
\newtheorem{lemma}[theorem]{Lemma}

\newtheorem{proposition}[theorem]{Proposition}
\newtheorem{remark}[theorem]{Remark}

\endlocaldefs

\begin{document}
	
	\begin{frontmatter}
		\title{On eigenvalues of a high-dimensional Kendall's rank correlation matrix with dependence}
		\runtitle{Kendall correlation matrix under dependence}
		
		\begin{aug}
				\author[A]{\fnms{Zeng} \snm{Li}\ead[label=e0]{liz9@sustech.edu.cn}},
			\author[B]{\fnms{Cheng} \snm{Wang}\ead[label=e1]{chengwang@sjtu.edu.cn}}
				\footnote{All authors contributed equally to this work and Cheng (\printead{e1}) is the corresponding author.},
						\and
			\author[C]{\fnms{Qinwen} \snm{Wang}\ead[label=e2]{wqw@fudan.edu.cn}}

					\address[A]{Department of Statistics and Data Science,
				Southern University of Science and Technology.
				\printead{e0}}
			
			\address[B]{School of Mathematical Sciences, MOE-LSC, Shanghai Jiao Tong University. 
				\printead{e1}}  
			
			\address[C]{School of Data Science, Fudan University.
				\printead{e2}}
	\end{aug}
	
		\begin{abstract}
				This paper investigates limiting spectral distribution of a high-dimensional Kendall's rank correlation matrix. The underlying population is allowed to have general dependence structure. The result no longer follows the generalized Mar\u{c}enko-Pastur law, which is  brand new. It's the first result on rank correlation matrices with dependence. As applications, we study the Kendall's rank correlation matrix for multivariate normal distributions with a general covariance matrix.  From these results, we further gain insights on Kendall's rank correlation matrix and its connections with the sample covariance/correlation matrix. 
		\end{abstract}
		
		\begin{keyword}[class=MSC2020]
			\kwd[Primary ]{62H12}
			\kwd{62G35}
			\kwd[; secondary ]{60F05}
		\end{keyword}
		
		\begin{keyword}
		\kwd{Hoeffding Decomposition} \kwd{Kendall's Rank Correlation Matrix} \kwd{Limiting Spectral Distribution}  \kwd{Random Matrix Theory}
		\end{keyword}
		
	\end{frontmatter}

\section{Introduction}
Covariance and correlation matrices play a vital role in multivariate statistical analysis because they provide the most direct way to characterize the relation between different variables. Many statistical estimation or inference methods involve the covariance or correlation matrix, such as principal component analysis, multivariate analysis of variance, factor analysis, etc.  In high-dimensional data analysis, studying the eigenvalues and eigenvectors of such covariance/correlation matrices are fundamental problems.

In random matrix theory,  sample covariance matrix has been thoroughly studied in the past decades.   For an $n \times n$ Hermitian matrix $\bH_n$, the \emph{Empirical Spectral Distribution} (ESD) of $\bH_n$ is defined as
\begin{align*}
	F^{\bH_n}(x)=\frac{1}{n}\sum_{i=1}^n I(\lambda_i\leq x),
\end{align*}
where $\lambda_1,\cdots,\lambda_n$ are eigenvalues of $\bH_n$ and $I(\cdot)$ is the indicator function. If $F^{\bH_n}$ converges to a deterministic distribution function $F$,  then $F(x)$ is called the \emph{Limiting Spectral Distribution} (LSD) of $\bH_n$.  \cite{marvcenko1967distribution} first derived the LSD of the sample covariance matrix. Further, \cite{bai2004clt} studied the \emph{Central Limit Theorem} (CLT) for its \emph{Linear Spectral Statistics} (LSSs) defined as
\begin{align*}
	\frac{1}{n} \sum_{i=1}^n f(\lambda_i)=\int f(x)  dF^{\bH_n}(x),
\end{align*}
where $f(\cdot)$ is a function on $\mR^+$. 
 As we know, many important statistics in multivariate analysis can be expressed through ESD, for example,
\begin{align*}
	\frac{1}{n}\tr(\bH_n)=\frac{1}{n} \sum_{i=1}^n \lambda_i=\int x dF^{\bH_n}(x)
\end{align*}
 and
 \begin{align*}
 		\frac{1}{n} \log|\bH_n|=\frac{1}{n}\sum_{i=1}^n \log{(\lambda_i)}=\int\log{(x)}  dF^{\bH_n}(x).
 \end{align*}
Generally, LSD  describes the first-order limits of these LSSs and CLT then  characterizes their second-order asymptotics. The two are the analogs of \emph{Law of Large Numbers} and \emph{Central Limit Theorem} in classical probability theory, respectively.  As applications,  CLT for LSSs provides an important tool for many hypothesis testing problems in multivariate analysis with diverging data dimension, e.g., \cite{bai2009corrections} derived the distribution of the likelihood ratio test for high-dimensional data. Both LSD and CLT for LSSs study the global law of the empirical eigenvalues. Another fundamental problem is the local law \citep{knowles2017anisotropic}, e.g., the asymptotic behaviors of the smallest and largest eigenvalues. The well-known Bai-Yin law  \citep{bai1993limit} derived the limits of the extreme eigenvalues. 
 \cite{johnstone2001distribution} further established the Tracy-Widom law for the largest eigenvalues, which plays a fundamental role in principal component analysis. 
 For more comprehensive overview on this topic, one is referred to \cite{bai2010spectral}.

In practice, data normalization is a standard procedure and after that, we are actually dealing with the sample correlation matrix \citep{el2009concentration}.  Parallel to the study of the sample covariance matrix,   \cite{jiang2004limiting} first obtained the LSD of Pearson-type sample correlation matrix and \cite{gao2017high} developed the CLT for its LSSs.  \cite{bao2012tracy} established the Tracy-Widom law for its extreme eigenvalues and \cite{pillai2012edge} extended the result to general cases. However, due to the complex structure of the sample correlation matrix,  most results only consider the case that sample data has independent components so that its population covariance matrix is diagonal  and then the correlation matrix is  identity.  From the perspective of applications, the independence assumption is however too strong so that such results have very limited applicability \citep{el2009concentration}. On the other hand, for general dependent  or correlated data, little work  \citep[e.g.,][]{el2009concentration, morales2021asymptotics} has been done on the sample correlation matrix. As far as the CLT for LSSs, the existing works include  \cite{mestre2017correlation} that considered  the case 
	for Gaussian distributions and \cite{zheng2019test} studied the trace moments.


For the sample covariance/correlation matrix,  due to the congenital sensitivity of Pearson-type correlation, finite fourth order or even higher order moments of the data distribution are usually required to guarantee the convergence of the limiting distributions. However, most of the results applicable to light-tailed distributions cannot be directly extended to heavy-tailed cases, e.g., \cite{heiny2020limiting} explored the spectral behavior of Pearson-type correlations for heavy-tail distributions where the story becomes completely different.


As a remedy for dealing with heavy-tailed data samples, some non-parametric correlation matrix, such as Kendall's $\tau$ and Spearman's $\rho$, have received considerable attention in recent years. 
Kendall's $\tau$ and Spearman's $\rho$ are rank-based  and thus there's no need to impose any moment restrictions on the underlying distribution. What is more,  classical theory on non-parametric statistics shows that only  partial information will be lost while robustness can be retained if we only use the ranks of the data. In random matrix theory,  \cite{bai2008large} first derived the LSD of Spearman's rank correlation matrix, which turns out to be the same as the standard Mar\u{c}enko-Pastur law. For the Kendall's $\tau$, \cite{bandeira2017marvcenko} proved that its LSD is an affine transformation of the standard Mar\u{c}enko-Pastur law.
For the CLT for LSSs,  \cite{bao2015Spearman} considered Spearman's rank correlation matrix and  \cite{ZWL2021Kendall} studied Kendall's rank correlation matrix. The Tracy-Widom law for the extreme eigenvalues of the two matrices can be found in \cite{bao2019tracy_sp} and \cite{bao2019tracy}, respectively.  However, all these asymptotic results are for data sample with independent components, i.e., all components are independent. To the best of our knowledge,  there are no available results on such rank correlation matrices when the underlying distribution has general dependent  structure.  We summarize the developments of the sample covariance matrix,  sample correlation matrix,  Kendall's $\tau$ and  Spearman's $\rho$ in Table \ref{table1}. 

\begin{table}[ht!]
	\centering
	\caption{Developments of sample covariance/correlation matrices in random matrix theory}
	\label{table1}
		\resizebox{12cm}{!}{%
	\begin{tabular}{|c|c|c|c|c|}
\hline
& \mbox{Sample~covariance}&	\mbox{Sample~correlation}&\mbox{Kendall's $\tau$}&	\mbox{ Spearman's $\rho$}\\
\hline
& \multicolumn{4}{c|}{\mbox{Independent case ($\bSig=\bI$)}}\\
\hline
\mbox{LSD}&\cite{marvcenko1967distribution}&\cite{jiang2004limiting}& \cite{bandeira2017marvcenko}&\cite{bai2008large}\\
\hline
\mbox{CLT for LSSs}  &\cite{bai2004clt}&\cite{gao2017high}&\cite{ZWL2021Kendall}& \cite{bao2015Spearman}\\
\hline
\mbox{Tracy-Widom}&\cite{johnstone2001distribution}&\cite{bao2012tracy}&\cite{bao2019tracy} &\cite{bao2019tracy_sp}\\
\hline
& \multicolumn{4}{c|}{\mbox{Dependent case (general $\bSig$)}}\\
\hline
\mbox{LSD}&\cite{marvcenko1967distribution} &\cite{el2009concentration}&&\\
\hline
\mbox{CLT for LSSs}&\cite{bai2004clt} &\cite{mestre2017correlation}&&\\
\hline
\mbox{Tracy-Widom}&\cite{feral2009largest} &&&\\
\hline
\cline{2-5}
\end{tabular}}
\end{table}

As can be seen from Table \ref{table1}, the asymptotic behaviors of the eigenvalues of the  rank  correlation matrices under general dependent structure is still unclear. 
In this paper, we take the first step to fill this gap and focus on Kendall's rank correlation matrix with high-dimensional correlated data. Specifically, for data sample $\bx_1,\cdots,\bx_n \in \mR^p$, we define the sign vector
\begin{align*}
	\A_{ij}=\sign(\bx_i-\bx_j)=\left(\sign(x_{i1}-x_{j1}),\cdots,\sign(x_{ip}-x_{jp})\right)\trans
\end{align*}
where $\sign(\cdot)$ denotes the sign function and the sample Kendall's rank correlation matrix  \citep{kendall1938new} 
\begin{align} \label{kn1}
	\K_n=\frac{2}{n(n-1)} \sum_{1\leq i<j\leq n} \A_{ij}\A\trans_{ij}.
\end{align}
Our goal is to study the spectral properties of $\K_n$ when $\bx_i's$ have a general dependent structure. This is a problem of its own significant interest in random matrix theory. To study Kendall's rank correlation matrix with dependence, LSD is the cornerstone for further derivations of CLT for LSSs \citep{bai2004clt} and local laws including the asymptotic distribution of extreme eigenvalues \citep{knowles2017anisotropic}.  It's also a key step to solve many high-dimensional statistical problems with heavy-tailed observations. Taking high-dimensional independent test as an example, many test statistics based on covariance/correlation matrices have been proposed to test complete independence among the components of $\bx_i's$; see \cite{schott2005testing}, \cite{bao2015Spearman}, \cite{gao2017high}, \cite{leung2018testing}, \cite{bao2019tracy}, \cite{ZWL2021Kendall} etc.  In particular, \cite{leung2018testing} and \cite{ZWL2021Kendall} considered test statistics that based on linear functions of the eigenvalues of Kendall's $\tau$, e.g., $\tr(\K_n^2)$ and $\log|\K_n|$. However, the test power is still unclear since the limiting properties of  Kendall's $\tau$  (also Spearman's $\rho$) under general dependent alternatives remain largely unknown. 

To answer such  questions, in this paper, we take the first step to derive the limiting spectral distribution of $\K_n$ under the asymptotic regime where $p,n\rightarrow\infty$ and $p/n\rightarrow c \in (0,\infty)$.  One major challenge is the nonlinear dependent structure among the sign-based summands $\A_{ij}$ of $\K_n$. Hence, we apply the Hoeffding decomposition to $\A_{ij}$ to locate the leading terms. In this way, we obtain the equation which the Stieltjes transform of the limiting spectral distribution of $\K_n$ satisfies. It's a brand new distribution which relies heavily on both  the covariance and conditional covariance structure of $\A_{ij}$. As illustration, we study the normal distribution where the Kendall's rank correlation has a specific relation with the Pearson's correlation and then we derive explicit LSDs for some cases with common dependent structure. Simulation experiments also lend full support to the accuracy of our theoretical results.

The rest of the paper is organized as follows. Section~\ref{sec:main results} introduces some preliminary knowledge on Kendall's rank correlation matrix
 and Hoeffding decomposition.  Section \ref{main:lsd} contains our main results on the LSD of Kendall's rank correlation matrix for correlated data. Section \ref{main:guass} considers the Gaussian distributions and Section \ref{sec:simu} collects all the numerical experiments.  Proofs of the main results are given in the Appendix.

\section{Background on Kendall's rank correlation matrix}\label{sec:main results}
\subsection{Kendall's rank correlation matrix}
From the definition of  Kendall's rank correlation matrix  \eqref{kn1},  we can write
\begin{align*}
	\K_n=\frac{2}{n(n-1)} \sum_{1\leq i<j\leq n} \sign\left(\frac{\bx_i-\bx_j}{\sqrt 2}\right)\sign\left(\frac{\bx_i-\bx_j}{\sqrt 2}\right)\trans,
\end{align*}
which looks  similar with the  sample covariance matrix. To be specific, the sample  covariance matrix of the data sample $\bx_1,\cdots, \bx_n \in \mR^p$ can be written as  an U-statistic of order two, i.e.,
\begin{align}\label{scm}
	\bS_n=\frac{1}{n-1}\sum_{i=1}^n(\bx_i-\bar{\bx}) (\bx_i-\bar{\bx})\trans=\frac{2}{n(n-1)} \sum_{1\leq i<j\leq n} \left(\frac{\bx_i-\bx_j}{\sqrt{2}}\right) \left(\frac{\bx_i-\bx_j}{\sqrt{2}} \right)\trans.
\end{align}
Despite this similarity in their forms, the inner structure of the two matrices does not follow the same pattern. The \emph{sign} function introduces non-linear correlation into the matrix $\K_n$ and it is nontrivial to analyze such correlation even for binary random variables. For example, \cite{esscher1924method} spent a lot of efforts to derive the variance of Kendall's rank correlation for bi-normal distributions.  As can be seen from  \cite{childs1967reduction}, it is already quite complicated to calculate the integral of  \emph{sign} function over fourth order even for normal distribution. Thus, for high-dimensional Kendall's rank correlation matrix, it is very challenging to study its asymptotic properties.

One appealing property of Kendall's rank correlation is that it is monotonically invariant \citep{weihs2018symmetric}.
\begin{proposition}[Monotonic Invariance] \label{mon-inv}
	For any strictly increasing monotonic functions  $f_j(\cdot), j=1,\ldots,p$,  $\K_n$ is invariant for monotonic component transformation 
	\begin{align*}
		\bx_i=(x_{i1},\cdots,x_{ip})\trans \mapsto \big( f_1(x_{i1}),\cdots,f_p(x_{ip}) \big)\trans.
	\end{align*}
\end{proposition}
The reason is that Kendall's rank correlation is rank-based and the monotonic transformation does not change the order statistics. One special case is that for any linear transformation of the data
\begin{align*}
	\bx_i=(x_{i1},\cdots,x_{ip})\trans \mapsto (\mu_1+\sigma_1 x_{i1}, \cdots, \mu_p+\sigma_p x_{i1} )\trans,
\end{align*}
their corresponding Kendall's rank correlations  remain unchanged. That is, Kendall's $\tau$ is a correlation matrix which is invariant to the location and scale. More importantly, for any distribution,  Kendall's rank  correlation always exists and this fact makes it an important tool to characterize data with heavy tails. 

In previous works on Kendall's rank correlation matrix for high-dimensional data (e.g., \citealt{bandeira2017marvcenko}, \citealt{leung2018testing}, \citealt{bao2019tracy}, \citealt{ZWL2021Kendall}), they assumed that all the components were independent with absolutely continuous density.  By the monotonic invariance, we can always transform each component into a standard normal distribution and all components are still independent. Thus, it can be formulated as that $\bx_1, \ldots, \bx_n$ are independent and identically distributed (i.i.d.) from a standard multivariate normal distribution $N(0,\bI_p)$, from which we can see that these results are very limited.  In this work, we consider more general cases where the data components are allowed to have dependence. 

\subsection{Hoeffding decomposition}
In this part, to deal with the nonlinear dependent structure of high-dimensional Kendall's rank correlation matrix, we apply Hoeffding decomposition to find out the leading terms. Specifically, denote $\A_i$ as the conditional expectation of $\A_{ij}$ given $\bx_i$, 
\begin{align}\label{ai}
	\A_i  \defby \E\{\sign(\bx_i-\bx)\mid \bx_i\},
\end{align}
the Hoeffding decomposition for $\A_{ij}$ can be written as,
\begin{align}\label{hod}
	\A_{ij}=\A_i-\A_j+\epsilon_{ij},
\end{align}
where 
\begin{align*}
	\epsilon_{ij}=\sign(\bx_i-\bx_j)-\E\{\sign(\bx_i-\bx_j)\mid \bx_i\}+\E\{\sign(\bx_i-\bx_j)\mid \bx_j\}.
\end{align*}

Throughout this paper, we assume that $\bx_1, \ldots, \bx_n$ are i.i.d. from a population with absolutely continuous density. Then, we have
\begin{align*}
	\E(\A_{ij})=\E(\A_i)=\E(\epsilon_{ij})=\ovec,
\end{align*}
and the covariance matrices of $\A_{ij},~\A_i$ and $\epsilon_{ij}$ exist. Specially, we denote 
\begin{align}
	\bSig_1 \defby	\cov(\A_{ij}),\quad \bSig_2 \defby \cov(\A_i),\quad \cov(\epsilon_{ij})=\bSig_1-2\bSig_2 \defby \bSig_3.
\end{align}

\subsection{Preliminary results}

With the Hoeffding decomposition of $\A_{ij}$  described in \eqref{hod},   the Kendall's rank correlation matrix $\K_n$ can  be decomposed accordingly,
\begin{align}\label{dek}
	\K_n
	&=\frac{2}{n(n-1)} \sum_{1\leq i<j\leq n}(\A_i-\A_j+\epsilon_{ij})(\A_i-\A_j+\epsilon_{ij})\trans \nonumber \\ \nonumber
		&= \frac{2}{n(n-1)} \sum_{1\leq i<j\leq n}(\A_i-\A_j)(\A_i-\A_j)\trans+\frac{2}{n(n-1)} \sum_{1\leq i<j\leq n}(\A_i-\A_j)\epsilon_{ij}\trans\\ \nonumber
		&\quad+\frac{2}{n(n-1)} \sum_{1\leq i<j\leq n}\epsilon_{ij}(\A_i-\A_j) \trans+\frac{2}{n(n-1)} \sum_{1\leq i<j\leq n}\epsilon_{ij} \epsilon_{ij}\trans\\ 
	&=\M_1+\M_2+\M_2\trans+\M_3,
\end{align}
where 
\begin{align*}
&\M_1\defby \frac{2}{n(n-1)} \sum_{1\leq i<j\leq n}(\A_i-\A_j)(\A_i-\A_j)\trans=\frac{2}{n-1}\sum_{i=1}^n(\A_i-\bar{\A}) (\A_i-\bar{\A}) \trans,	\nonumber
\\
&\M_2\defby \frac{2}{n(n-1)} \sum_{1\leq i<j\leq n}(\A_i-\A_j)\epsilon_{ij}\trans,\quad \M_3\defby \frac{2}{n(n-1)} \sum_{1\leq i<j\leq n}\epsilon_{ij} \epsilon_{ij}\trans, 
\end{align*}
and $\bar{\A}=\frac 1n\sum_{i=1}^n \A_i$. In particular, $\M_1$ is the sample covariance matrix formed by  i.i.d. random vectors $\A_1,\ldots,\A_n$, see the illustration in \eqref{scm}.

Our first result is to show that the error terms $\M_2$ and $\M_3-\E(\M_3)$ can be controlled so that the dominant contribution   in terms of the LSD of $\K_n$  is from $\M_1$.
As a result, we define the following matrix 
\begin{align}\label{wn}
	\W_n \defby \M_1+\bSig_3= \frac{2}{n-1}\sum_{i=1}^n(\A_i-\bar{\A}) (\A_i-\bar{\A}) \trans+\bSig_3.
\end{align}
Throughout the paper, we use $\|\cdot\|$ and $\|\cdot \|_2$ to denote the common spectral norm and Frobenius norm of a matrix, respectively.

\begin{proposition}\label{thm1}
Assume $\|\bSig_1\| \leq C$ for some universal constant $C$ and 
\begin{align*}
	\frac{1}{p^2} \var\left(\A_{12} \trans \A_{13} \right) \to 0,
\end{align*}	 
then we have
\begin{align*}
	L(F^{\K_n}, F^{\W_n}) \to 0,~\mbox{in~probability}
	\end{align*}
	where $L(\cdot,\cdot)$ is the Levy distance between two distributions.
\end{proposition}

\begin{remark}
The assumptions on $\|\bSig_1\|$ and $\var\left(\A_{12} \trans \A_{13} \right)$ are to avoid too strong dependence among the components of the data. In random matrix theory, it is a regular condition to assume that the norm of the population covariance matrix is uniformly bounded, e.g.,  Condition 3 in \cite{bai2008large}. Here, this condition is also required for  bounding the difference between $\mathbf{K}_n$ and  $\mathbf{W}_n$. To show that such condition  $\|\bSig_1\| \leq C$ is necessary, we conduct a toy example in the following. Specifically,  we generate $n$ data sample $\bx_1,\ldots,\bx_n,i.i.d \sim N(0,\bSig)$ where $\bSig$ is a matrix with $\Sigma_{ii}=1$ and $\Sigma_{ij}=\rho$. For this case, 
	\begin{align*}
\|\bSig_1\| =1+\frac{2}{\pi}(p-1) \arcsin(\rho),
	\end{align*}
which is unbounded for any $\rho>0$. Figure \ref{fig1} presents the distance $\|\K_n-\W_n\|_2^2/p$ versus the increasing $\rho$, and from which we can see that $\|\bSig_1\| \leq C$ is necessary for bounding the difference between $F^{\K_n}$ and $F^{\W_n}$.
\begin{figure}[!h]
		\includegraphics[width=.7\linewidth]{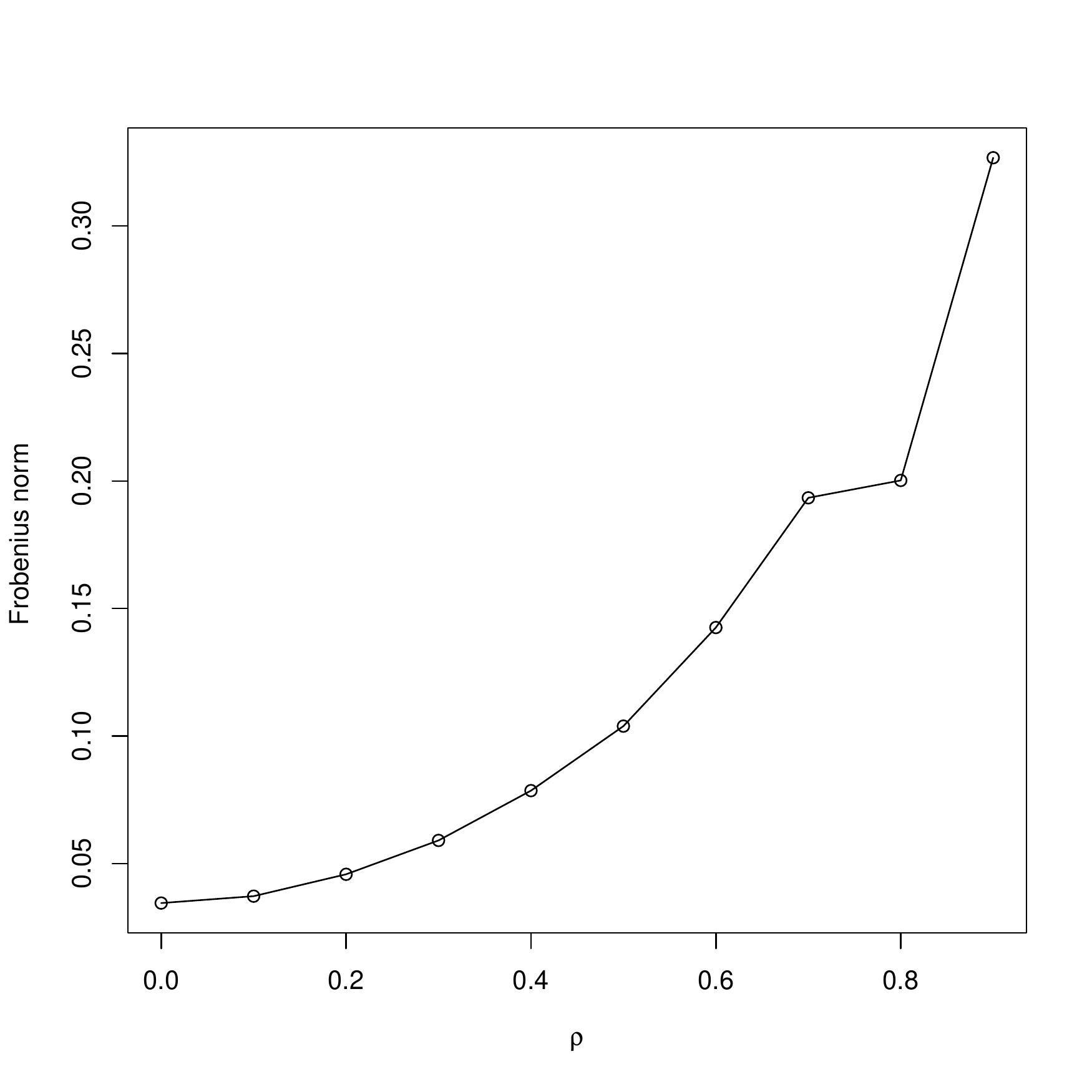}
	\caption{Plots of the scaled squared Frobenius norm of $\K_n-\W_n$  versus $\rho \in [0,0.9]$. Data sample $\bx_1,\cdots,\bx_n$ are generated from a multivariate normal distribution $N_p(0,\bSig)$ with $\Sigma_{ii}=1$ and $\Sigma_{ij}=\rho$. Here $(n,p)=(100, 200)$ and the results are based on 100 replications.}
	\label{fig1}
\end{figure}
\end{remark}

Noting that $\A_1,\ldots, \A_n$ are i.i.d. random vectors with covariance matrix $\cov(\A_i)=\bSig_2$ and 
	\begin{align*}
		\E \K_n=\bSig_1=2 \bSig_2+\bSig_3.
	\end{align*}
Proposition \ref{thm1} shows that part of Kendall's rank correlation matrix has similar fluctuations as the usual sample covariance matrix with population covariance matrix $2\bSig_2$ and the other part is concentrated on the  deterministic matrix $\bSig_3$. This phenomenon is an analogy of Hoeffding decomposition for the classical U-statistics. By implementing the Hoeffding decomposition for the random vector $\A_{ij}$, we then transfer the study of the  LSD of  $\K_n$ to the study of the LSD of $\W_n$.

\section{Limiting spectral distribution of  $\K_n$} \label{main:lsd}
In this section, we present the LSD of the  Kendall's rank correlation matrix $\K_n$. We first introduce the concept of 
 \emph{Stieltjes transform}, which is an important tool  in random matrix theory.
Letting $\mu$ be a finite measure on $\mR$, its \emph{Stieltjes transform} $s_{\mu}(z)$ is defined as 
\begin{align*}
	s_{\mu}(z)=\int\frac{1}{x-z} \mu(dx),~z\in \mC^{+},
\end{align*}
where $\mC^{+}$ denotes the upper complex plane. We can also obtain $\mu$ from $s_{\mu}(z)$ by the inversion formula. For any two continuity points $a<b$ of $\mu$, we have
\begin{align}\label{eq:inv}
	\mu([a,b])=\lim_{\nu \to 0^+}\frac{1}{\pi}\int_a^b \Im s_{\mu}(x+i\nu)dx,
\end{align}
where $\Im$ is the imaginary part of a complex number and $i$ is the imaginary unit. 

For Kendall's rank correlation matrix, \cite{bandeira2017marvcenko} derived the LSD when the observations $\bx_1, \ldots, \bx_n$ are i.i.d. random vectors and the components are also independent with absolutely continuous density. They show that as $n\rightarrow\infty$, $p/n\rightarrow c\in(0,\infty)$, the ESD of such Kendall's rank converges in probability to an affine transformation of  the standard Mar$\check{\mbox{c}}$enko-Pastur law with parameter $c$, which  has an explicit form whose density function $p_c(x)$ is given by
\begin{align*}
p_c(x)=
\frac{9}{4\pi c(3x-1)}\sqrt{(c_{+}-x)(x-c_{-})}+(1-1/c)\delta_{\frac{1}{3}}I(c>1),~c_{-} \leq x\leq c_{+},
\end{align*}
where $c_{-}=\frac{1}{3}+\frac{2}{3}(1-\sqrt{c})^2$ and $c_{+}=\frac{1}{3}+\frac{2}{3}(1+\sqrt{c})^2$. The corresponding Stieltjes transform $s(z)\in \mC^{+}$ is the unique solution to the following equation
\begin{align}\label{eq:INDszeq}
	\frac{2}{3}c\left(z-\frac{1}{3}\right)s^2(z)+\left(z-1+\frac{2}{3}c\right)s(z)+1=0.
\end{align}

To illustrate the challenges of Kendall's rank correlation matrix in random matrix theory, we consider the ranking of the data
\begin{align*}
\begin{pmatrix}
\bx_1\trans\\
\vdots\\
\bx_n \trans
\end{pmatrix}=\underbrace{\begin{pmatrix}
x_{11} &\cdots&x_{1p}\\
 \vdots &\vdots&\vdots\\
  x_{n1} &\cdots&x_{np}\\
	\end{pmatrix}}_{\text{raw data matrix}} \Longrightarrow \underbrace{\ \begin{pmatrix}
r_{11} &\cdots&r_{1p}\\
	\vdots &\vdots&\vdots\\
	r_{n1} &\cdots&r_{np}\\
\end{pmatrix}}_{\text{ranking matrix}},
\end{align*}
where each column $(r_{1j},\ldots,r_{nj})$ are the rank of the raw data $(x_{1j},\ldots,x_{nj})$.
 For i.i.d. sample $\bx_1,\cdots, \bx_n\in \mR^p$,  each column of the ranking matrix follows the uniform distribution on the set of all $n!$ permutations of $\{1,2,\cdots,n\}$. While the rows of the raw data matrix are independent, the ranking matrix dose not have independent rows anymore.   For the special case where the columns of the raw data are also independent (e.g., \citealt{bandeira2017marvcenko}, \citealt{leung2018testing}, \citealt{bao2019tracy}, \citealt{ZWL2021Kendall}), the columns of the ranking matrix will be independent. Then, the raw data is actually with i.i.d entries which has very limited applications.  If the columns of the raw data are dependent, e.g., there is a covariance structure among components, both the rows and the columns of the ranking matrix are dependent. From the perspective of random matrix theory, analyzing such matrices is very challenging.
 
We first provide a general result as follows and then study the case for Gaussian distribution in the next section.
\begin{theorem}\label{thm:thlsd}
	For i.i.d. continuous data sample $\bx_1,\cdots, \bx_n\in \mR^p$, assume that
	\begin{itemize}
		\item[(A)] 	as $p \to \infty$,
		\begin{align} \label{ass2}
		\frac{1}{p^2} \var\left(\A_{12} \trans \A_{13} \right) \to 0, \text{and}~~	\frac{1}{p^2}	\var(\A_1 \trans \B \A_1)\to 0,
		\end{align}
	where $\B$ is any deterministic matrix with bounded spectral norm;
		\item[(B)] $\|\bSig_1\| \leq C$ for some universal constant $C$, also the solution $x(z) \in \mathbb C^-$ to the following equation exists
		\begin{align}\label{eq:x}
			\frac {1}{x(z)}=1+\lim_{n\to \infty}\frac{2}{n}\tr \big[(\bSig_3+2x(z) \bSig_2-z\bI_p)^{-1}\bSig_2\big];
		\end{align}
		\item[(C)] $p,n\rightarrow \infty$ such that $p/n\rightarrow c\in (0,\infty)$.
	\end{itemize}
	
	Then, in probability, the
	empirical spectral distribution $F^{\K_n}$ converges weakly to a
	limiting spectral distribution  $F$ whose Stieltjes transform
	$s(z)$ is given by 
\begin{align}\label{lsd}
 s(z)=\lim_{p\to \infty}\frac{1}{p}\tr \big[(\bSig_3+2x(z) \bSig_2-z\bI_p)^{-1}\big].
\end{align}
	
\end{theorem}
Recall the dominating matrix $\W_n$ in \eqref{wn},
which has the same LSD as the following matrix
\begin{align}\label{m2}
	\frac{2}{n}\sum_{i=1}^n \A_i \A\trans_i+\bSig_3.
\end{align}
The first part $ n^{-1}\sum_{i=1}^n \A_i \A\trans_i$ is the type of a sample covariance matrix  corresponding to the data sample $\{\A_i\}$ with population covariance matrix $\cov(\A_i)=\bSig_2$. However, the  components within  $\A_i$ are  nonlinearly correlated and thus can not be written in the form of independent components model such that $\A_i=\bSig^{1/2}_2\bz_i$. For those weakly dependent  data sample, \cite{bai2008large} proved that under certain conditions, the LSD of the sample covariance matrix  still follows  the generalized Mar\u{c}enko-Pastur law. One of the crucial conditions is that the variance of the quadratic forms $\A_i \trans \B \A_i$ is relatively small (see Theorem 1.1 in \citealt{bai2008large}), i.e.,
		\begin{align*}
			\var(\A_i \trans \B \A_i) =o(n^2),
		\end{align*}
 which is actually the second part of our assumption (A). Under this condition, the LSD of the first part  $n^{-1}\sum_{i=1}^n \A_i \A\trans_i$ remains the same as the generalized Mar\u{c}enko-Pastur law corresponding to the population covariance matrix $\bSig_2$.

On the other hand, the LSD of Hermitian matrix of the type $\X\bT\X\trans/n+\A$ has been studied in \cite{silverstein1995empirical} where $\X$ is assumed to be an $n\times p$ random matrix with  i.i.d. standardized entries, $\bT$ is a diagonal matrix having an LSD, $\A$ is an Hermitian matrix and the three matrices are independent. Under certain conditions, \cite{silverstein1995empirical}  proved that  the LSD of $\X\bT\X\trans/n+\A$ 
is a shift  of the LSD of $\A$. Intuitively, this is because the population version of $\X\bT\X\trans$ equals $(\tr \bT)  \bI_n$, which shares the same eigenvectors  as the matrix $\A$. However, in our case, the population version of the first  parts $n^{-1}\sum_{i=1}^n \A_i \A\trans_i$ in \eqref{m2} equals $\bSig_2$, whose eigenvectors might be different from the ones of $\bSig_3$. 
Therefore, we can not directly apply the results in  \cite{silverstein1995empirical}.  Using the  terminology of matrix subordination \citep{kargin2015subordination}, the limiting Stieltjes transform of $\X\bT\X\trans+\A$ is subordinated to the limiting Stieltjes transform of $\A$.  Roughly, our result  is the same as the model $2\bSig_2^{1/2} \X\trans \X  \bSig_2^{1/2}/n + \bSig_3$.  Our Theorem \ref{thm:thlsd} shows that
both the eigenvalues and eigenvectors of the  two matrices, $\bSig_2$ and $\bSig_3$, will contribute to the LSD of $\K_n$. Thus, it's a brand new LSD for covariance/correlation matrix.

Technically, the assumption (B) is a new condition and in Appendix, we prove the uniqueness of
$s(z)$ or $x(z)$ if it exists. Here we make some discussions on this condition. If $\bSig_3=\mathbf{0}$, the equation \eqref{eq:x} will be 
\begin{align*}
\lim_{p \to \infty}\frac{1}{p}\tr \big[(2x(z) \bSig_2-z\bI_p)^{-1}\big]=\frac{1-c-x}{cz},
\end{align*}
which means that the limiting Stieltjes transform of $\bSig_2$ exists and we solve the above equation to get $x(z)$. The Stieltjes transform of the LSD is then 
\begin{align*}
	s(z)=\lim_{p\to \infty}\frac{1}{p}\tr \big[(2x(z) \bSig_2-z\bI_p)^{-1}\big]=\frac{1-c-x}{cz},
\end{align*} 
and then we can get 
\begin{align*}
	-\frac{x}{z}=-\frac{1-c}{z}+c\cdot s(z),
\end{align*}
where the right hand side is exactly the limiting Stieltjes transform of $ \X  \bSig_2\X\trans /n$.  Thus, our result under the special case $\bSig_3=\mathbf{0}$ is consistent with the one in \cite{bai2008large}. Now we consider another special case that $\bSig_2=\bI_p/2$. From  \eqref{eq:x}  and \eqref{lsd}, we can get 
	\begin{align*}
 s(z)=\lim_{p \to \infty}\frac{1}{p}\tr \big[(\bSig_3+(x-z)\bI_p)^{-1}\big]=\frac{1-x}{xc},
\end{align*}
which yields 
\begin{align*}
x=\frac{1}{1+c\cdot s(z)}.
\end{align*}
This result is consistent with the one in \cite{silverstein1995empirical} when $\bT=\bI$. 

In summary, our new LSD extends the results in \cite{silverstein1995empirical} and \cite{bai2008large}.  For general $\bSig_2$ and $\bSig_3$, it is challenging to study the limits of \eqref{eq:x}. One special case is that $\bSig_2$ and $\bSig_3$ are simultaneously diagonalizable and  Toeplitz matrix is such an example, which will be studied in the next section.


\section{Gaussian ensemble}\label{main:guass}
As mentioned in the introduction, the Pearson correlation matrix has been thoroughly studied in random matrix theory. Generally, there is no explicit relation between Kendall's correlation and Pearson correlation; see  \cite{kendall1949rank} for more details. A special ensemble is Gaussian distribution  where  
Kendall's correlation has a monotonic correspondence with Pearson correlation.  This neat relation is presented in the following lemma  which is called Grothendieck's Identity in mathematical community.

	\begin{lemma}[Grothendieck's Identity] \label{clalem1}
	Consider a  bi-variate normal distribution:
	\begin{align*}
		\begin{pmatrix}
			z_1\\
			z_2
		\end{pmatrix} \sim N \left\{\begin{pmatrix}
			0\\
			0
		\end{pmatrix}, \left(\begin{array}{cc}
			1&  ~\rho \\
			\rho &~1
		\end{array}\right) \Large\right\},
	\end{align*}
	where $\rho \in [-1,1]$. We have
	\begin{align*}
		\E\left\{\sign(z_1)\sign(z_2)\right\}
		=4 \E \left\{  I(z_1, z_2>0)\right\}-1= \frac{2}{\pi} \arcsin{\rho}.
	\end{align*}
\end{lemma}
Assume $\bx_1,\ldots,\bx_n,i.i.d \sim N(0,\bSig)$ where $\bSig$ is a correlation matrix, by Lemma \ref{clalem1}, we can show that
\begin{align}\label{s1s2s3}
	\bSig_1=&\frac{2}{\pi} \arcsin(\bSig),~\bSig_2=\frac{2}{\pi} \arcsin(\bSig/2), \nonumber\\
	~\bSig_3=& \frac{2}{\pi} \arcsin(\bSig)-\frac{4}{\pi} \arcsin(\bSig/2).
\end{align}
Thus, for Gaussian distribution, Kendall's rank correlation matrix is determined by the Pearson's correlation matrix $\bSig$. 

In this section, we consider the LSD of $\K_n$ for Gaussian ensembles from which can shed new light on Kendall's correlation matrix  and also its connections with the sample covariance/correlation matrix.
\begin{proposition} \label{thmnorml}
	Assume $\bx_1,\ldots,\bx_n,i.i.d \sim N(0,\bSig)$ where $\bSig$ is a correlation matrix. Under the Assumptions (B) and (C) in Theorem \ref{thm:thlsd}, the conclusion of Theorem \ref{thm:thlsd} holds.
\end{proposition}

\begin{remark}
	It is noted that although we consider the Gaussian ensemble, the results actually cover a wider range of distributions, which is called non-paranormal distribution \citep{liu2009nonparanormal} due to the monotonic invariance of Kendall's rank correlation matrix.
	To be specific, a random vector $Y=(Y_1,\cdots,Y_p)\trans \in \mR^p$ is said to have a non-paranormal distribution if there exist monotone functions $\left\{f_j\right\}_{j=1}^p$ such that $\left( f_1(Y_1),\cdots,f_p(Y_p)\right) \sim N(\mu,\bSig)$.

\end{remark}

The proof of Proposition \ref{thmnorml} is to check the assumption \eqref{ass2} for Gaussian distribution.  Specially, for the normal distribution or non-paranormal distribution, we  can calculate the variance of $\A_{12} \trans \A_{13}$ explicitly which is based on the classical results in \cite{esscher1924method} and control the variance of the quadratic form $\A_1 \trans \B \A_1$ using Poincar\'{e} inequality. Hence,  Assumption (A) holds for Gaussian distribution and the detailed proof is presented in Appendix. Next, we consider some examples to illustrate the result.

\subsection{Independent case}
	A very special case is the standard multivariate normal distribution, i.e., $\bSig=\bI_p$.  By the monotonic invariance of Kendall's rank correlation matrix, it is equivalent to the independent case considered by \cite{bandeira2017marvcenko}, \cite{leung2018testing}, \cite{bao2019tracy}, and \cite{ZWL2021Kendall}.

When $\bSig=\bI_p$, we know $\bSig_2=\bSig_3=\frac{1}{3} \bI_p$. Intuitively, the matrix given in \eqref{m2} reduces to a standard sample covariance matrix  corresponding to the population covariance matrix $\bSig_2=\frac{2}{3} \bI_p$ and the deterministic matrix $\bSig_3=\frac{1}{3} \bI_p$. This explains that its LSD is  $\frac{2}{3}\text{MP}+\frac{1}{3}$. As an illustration of our main theorems,  we demonstrate this result using our Theorem~\ref{thm:thlsd} and Proposition \ref{thmnorml} in the following.

Starting from  equation~\eqref{eq:x}, we have 
\begin{align*}
		\frac{1}{x(z)}=1+\frac{2c}{1+2x(z)-3z},
\end{align*}
	and this equation has a unique solution in $\mathbb{C}^{-}$, 
\begin{align*}
	x(z)=\frac {1}{4}\Big\{1-2c+3z-\sqrt{(2y-3z-1)^2-8(3z-1)}\Big\}.
\end{align*}
	Plugging it into \eqref{lsd}, we obtain that
\begin{align*}
	s(z)=\frac{1-\frac{2}{3}c-z+\sqrt{(z-1-\frac{2}{3}c)^2-\frac{16}{9}c}}{\frac{4}{3}c(z-\frac{1}{3})},
\end{align*}
	which is the Stieltjes transform of $\frac{2}{3}\text{MP}+\frac{1}{3}$  as shown in \cite{bandeira2017marvcenko}.

\subsection{MA(1) model}
Next, we consider an MA(1) model with population correlation matrix  $\bSig$  as follows,
\begin{align*}
\bSig=\bSig(\rho)=\begin{pmatrix}
		1 & \rho &  0&\cdots&0\\
		\rho & 1 & \rho&\ddots&0\\
		0&\ddots&\ddots&\ddots&0\\
		\vdots&&\rho&1 &\rho\\
	0&\cdots&0&\rho&1
	\end{pmatrix},
\end{align*}
where $\rho \in (-1/2,1/2)$. The eigenvalues of $\bSig(\rho)$ are given by
\begin{align*}
\lambda_k(\rho)=1+2 \rho \cos\frac{k \pi}{p+1},~k=1,\ldots,p
\end{align*}
and  the corresponding eigenvectors are 
\begin{align*}
	\bu_k=\sqrt{\frac{2}{p+1}}\bigg(\sin\frac{k \pi}{p+1},\sin\frac{2 k \pi}{p+1}, \cdots, \sin\frac{p k \pi}{p+1}\bigg) \trans.
\end{align*}
A detailed calculation of the eigenvalues and eigenvectors  can be found in Lemma 1 of  \cite{wang2011limiting}. It is noted that the eigenvectors  of $\bSig(\rho)$ do not depend on the correlation parameter $\rho$. Thus, $\bSig_2$ and $\bSig_3$ share the same eigenvectors and we can derive the two limits of Theorem \ref{thm:thlsd} as follows.

\begin{proposition} \label{case2}
Assume that $\bx_1,\ldots, \bx_n,i.i.d. \sim N \left(0,\bSig(\rho)\right)$. Then the Stieltjes transform $s(z)$ of the LSD of $\K_n$ satisfies
\begin{align}\label{lm2}
	s(z)=-\frac{1}{\sqrt{\Big(\frac{1}{3}+\frac{2x(z)}{3}-z\Big)^2-4\Big(\frac{2}{\pi} \arcsin {\rho}+\frac{4(x(z)-1)}{\pi} \arcsin\frac{\rho}{2}\Big)^2}}.
	\end{align}
Here $s(z)\in \mathbb{C}^+$ and $x(z)\in \mathbb{C}^-$ satisfies	
\begin{align}\label{ma}
	\frac{1}{2c}\left(\frac{1}{x(z)}-1\right)=\frac{1-c(x(z),\rho)(1+2x(z)-3z)}{3} s(z)+c(x(z),\rho)
\end{align}

where 
\begin{align*}
c(x,\rho)=	 \left\{ \begin{array}{ll}\displaystyle
		\frac{\arcsin\frac{\rho}{2}}{\arcsin {\rho}+2(x-1)\arcsin\frac{\rho}{2}},~~ & \mbox{if $\rho \neq 0$},\\
		0 & \mbox{if $\rho=0$}.\end{array} \right.
\end{align*}
\end{proposition}
By solving  \eqref{lm2} and \eqref{ma} in Proposition \ref{case2}, we can derive the Stieltjes transform $s(z)$ of the LSD of $\K_n$ when samples are from a MA(1) model.

\subsection{Toeplitz structure}
Last but not least, we consider a more general case that  the population  correlation matrix $\bSig$ has a Toeplitz structure
\begin{align}\label{top}
	\bSig=\begin{pmatrix}
		1 & \rho_1 & \rho_2& \cdots & \rho_{p-1}\\
	\rho_1& 1 & \rho_1  &\cdots&\rho_{p-2} \\
 \vdots & \ddots & \ddots & \ddots&\vdots\\
  \rho_{p-2}&\cdots& \rho_1 & 1& \rho_1\\
\rho_{p-1}&\cdots& \cdots & \rho_1& 1\\
	\end{pmatrix},
\end{align} 
where the correlations are absolutely summable
\begin{align} \label{ab-sum}
	\sum_{k=1}^{\infty} |\rho_k| < \infty.
\end{align}
Define the function 
\begin{align*}
	f(\theta)=1+\sum_{k=1}^{\infty} \rho_k(e^{i k\theta}+e^{-i k\theta})=1+2 \sum_{k=1}^{\infty} \rho_k\cos{(k\theta)},
\end{align*}
whose Fourier series are exactly $(1,\rho_1,\cdots)$. Szeg\"{o}  Theorem \citep{gray2006toeplitz} shows that the eigenvalues of $\bSig$ can be approximated by
\begin{align*}
	f\left(\frac{k \pi}{p+1}\right),~k=1,\cdots,p.
\end{align*}
For the limit in \eqref{eq:x}, which involves two Toeplitz matrices, we can not apply  Szeg\"{o}  Theorem directly.   However, a Toeplitz matrix can be approximated by a circulant matrix    \citep[Lemma 11]{gray2006toeplitz} whose eigenvectors are universal for its entries. By Theorems 11 and 12 of \cite{gray2006toeplitz}, under some mild conditions, we have the following limit.

\begin{proposition}\label{case3}
	For the Toeplitz matrix $\bSig$ defined in \eqref{top}, we have
	\begin{align}\label{eq:Toep}
		\lim_{p \to \infty}\frac{1}{p}\tr \big[(\bSig_3+2x \bSig_2-z\bI_p)^{-1}\bSig_2\big]=\frac{1}{2\pi}\int_{0}^{2\pi} \frac{f_2(\theta)}{f_1(\theta)} d\theta,
	\end{align}
	where 
	\begin{align*}
		f_1(\theta)=&\frac{1}{3}+\frac{2x}{3}-z+2 \sum_{k=1}^{\infty} \Big\{\frac{2}{\pi} \arcsin {\rho_k}+\frac{4(x-1)}{\pi} \arcsin\frac{\rho_k}{2}\Big\}\cos{(k\theta)},\\
		f_2(\theta)=&\frac{1}{3}+\frac{4}{\pi} \sum_{k=1}^{\infty}  \arcsin\frac{\rho_k}{2} \cos{(k\theta)}.
	\end{align*}
\end{proposition}
The absolutely summable condition \eqref{ab-sum} and the bound for $\arcsin(\cdot)$ guarantee the existence of the Fourier functions $f_1(\cdot)$ and $f_2(\cdot)$.  By solving \eqref{eq:x} using the limit \eqref{eq:Toep} in Proposition \ref{case3}, we can theoretically derive the $x(z)$. For the  limit \eqref{lsd} in Theorem \ref{thm:thlsd}, Szeg\"{o}  Theorem can yields the result directly. In summary,  we can obtain the Stieltjes transform $s(z)$ of the LSD of $\K_n$ when samples are from the Toeplitz covariance matrix model as follows. 
\begin{proposition}
Assume that $\bx_1,\ldots, \bx_n,i.i.d. \sim N \left(0,\bSig\right)$ where $\bSig$ is a Toeplitz matrix \eqref{top}. Then the Stieltjes transform $s(z)$ of the LSD of $\K_n$ satisfies
\begin{align*}
s(z)=\lim_{p \to \infty} \frac{1}{p}\tr (\bSig_3+2x \bSig_2-z\bI_p)^{-1}=\frac{1}{2\pi}\int_{0}^{2\pi} \frac{1}{f_1(\theta)} d\theta,
\end{align*}
where 
\begin{align*}
	\frac{1}{x(z)}=1+\frac{c}{\pi}\int_{0}^{2\pi} \frac{f_2(\theta)}{f_1(\theta)} d\theta.
\end{align*}
\end{proposition}

\section{Simulation}\label{sec:simu}
In this section, simulation experiments are conducted to examine the finite sample performance of eigenvalues of Kendall's sample correlation matrix when the data sample follows different dependence structure. We generate sample data $\bx_1,
\cdots, \bx_n\sim N_p(0,\bSig_0)$, draw the histogram of eigenvalues of the Kendall sample correlation matrix and compare with their theoretical densities. Specifically, we consider four types of covariance matrix $\bSig_0$:
\begin{itemize}
	\item[(I)] Independent case: $\bSig_0=\bSig=\bI_p$;
	\item[(II)]  Factor mode: $\bx_t=\A f_t+\varepsilon_t$, where $f_t\sim N_k(0,\bI_k)$, $\varepsilon_t\sim N_p(0,\bI_p)$, thus $\cov(\bx_t)=\bSig_0=\bI_p+\A \trans \A$, where $\mbox{rank}(\A)=k=3$;
	\item[(III)] MA(1) model: all the diagonal entries of $\bSig_0$ are 1, both upper and lower subdiagonal entries are $\rho$, others are zero;
	\item[(IV)] General Toeplitz matrix with $\rho_1=\rho_2=\rho,~\rho_k=0$ for $k\geq 3$.	
\end{itemize}
Here the population covariance matrix $\bSig_0$ and correlation matrix $\bSig$ are the same in Model (I), (III) and (IV).

\vspace{0.2cm}
\noindent
\textbf{Model (I): }\\
As for the independent case, we consider three types of sample correlation matrices, Pearson $\bR_n$, Spearman $\bS_n$ and Kendall $\K_n$. Specifically, for our data sample $\X_n=(\bx_1,\cdots,\bx_n)_{p\times n}$, $\bx_i=(x_{i1},\cdots,x_{ip})_{p\times1}$,  both $\bS_n=\lb
s_{k\ell}\rb$ and $\bR_n=\lb \rho_{k\ell}\rb$ are
$p\times p$ matrices where $s_{k\ell}$ and $\rho_{k\ell}$ are the Spearman and
Pearson correlation of the $k$-th and $\ell$-th row of $\X_n$ with
\begin{gather*}
	s_{k\ell}=\frac{\sum_{i=1}^n\lb r_{ki}-\overline{r}_k \rb\lb r_{\ell i}-\overline{r}_\ell \rb}{\sqrt{\sum_{i=1}^n\lb r_{ki}-\overline{r}_k\rb^2}\sqrt{\sum_{i=1}^n\lb r_{\ell i}-\overline{r}_\ell \rb^2}},\quad  \overline{r}_k=\frac{1}{n}\sum_{i=1}^n r_{ki}=\frac{n+1}{2},\\
	\rho_{k\ell}=\frac{\sum_{i=1}^n\lb x_{ki}-\overline{x}_k \rb\lb x_{\ell i}-\overline{x}_\ell \rb}{\sqrt{\sum_{i=1}^n\lb x_{ki}-\overline{x}_k\rb^2}\sqrt{\sum_{i=1}^n\lb x_{\ell i}-\overline{x}_\ell \rb^2}},\quad \overline{x}_k=\frac{1}{n}\sum_{i=1}^n x_{ki},
\end{gather*}
here $r_{ki}$
is the rank of $x_{ki}$ among $\lb x_{k1},\cdots, x_{kn}\rb$. From \cite{jiang2004limiting} and \cite{bai2008large}, we know that the LSD of $\bS_n$ and $\bR_n$ are both standard Mar\u{c}enko-Pastur law while the LSD of $\K_n$ is an affine transformation of Mar\u{c}enko-Pastur law \citep{bandeira2017marvcenko}. Thus we list the histogram of eigenvalues of all three types of sample correlation matrices  under different combinations of $(p,n)$ and compare with their corresponding limiting densities in Figure~\ref{fig:indep}. It can be seen from Figure~\ref{fig:indep} that all the histograms conform to their theoretical limits, which fully supports our theoretical results in the independent case.
\begin{figure}[!h]
\includegraphics[width=\linewidth]{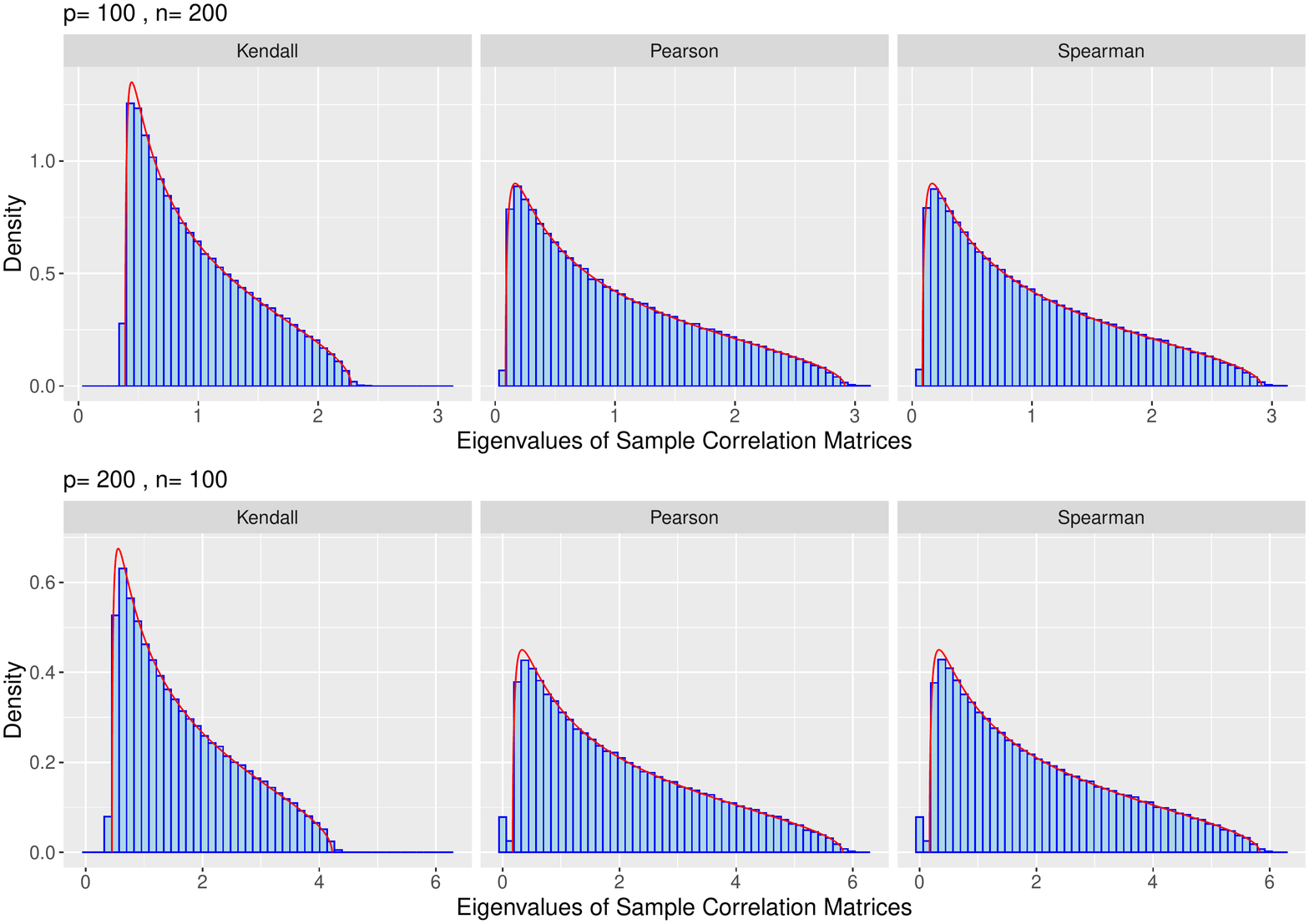}
	\caption{Histograms of eigenvalues of three types of sample correlation matrices, Kendall $\K_n$, Pearson $\bR_n$, Spearman $\bS_n$ for data sample $\bx_1,\cdots,\bx_n\sim N_p(0,\bI_p)$ with $(p,n)=(100, 200)$ and $(p,n)=(200, 100)$. The red curves are density functions of their corresponding limiting spectral distribution.}
	\label{fig:indep}
\end{figure}

\vspace{0.2cm}
\noindent\textbf{Model (II): }\\
The second case is the factor model or the spiked model, i.e., the population covariance matrix is
	\begin{align*}
		\cov(\bx_t)=\bSig_0=\mathbf{I}_p+\A \trans \A,
	\end{align*}
	where $\mbox{rank}(\A) \leq k$. For the related correlation matrix $\bSig$, we have 
	\begin{align*}
		\bSig=\di(\bSig_0)^{-1/2} \cdot \bSig_0 \cdot \di(\bSig_0)^{-1/2} =\di(\bSig_0)^{-1}+\tilde{\A} \trans \tilde{\A},
	\end{align*}
	where $\tilde{\A}= \A \di(\bSig_0)^{-1/2}$. Noting,
	\begin{align*}
		\bSig_1=&\frac{2}{\pi} \arcsin(\bSig),~\bSig_2=\frac{2}{\pi} \arcsin(\bSig/2), ~\bSig_3=\bSig_1-2\bSig_2,
	\end{align*}
	and $2 x/\pi \leq 2 \arcsin(x)/\pi \leq x$ for any $x \in [0,1]$, we have
	\begin{align*}
		\frac{1}{p} \left\|	\bSig_1-\mathbf{I}_p\right\|_2^2 \leq 	\frac{1}{p} \left\|\frac{2}{\pi} \arcsin(\tilde{\A} \trans \tilde{\A}) \right\|_2^2 \leq 	\frac{1}{p} \left\|\tilde{\A} \trans \tilde{\A}\right\|_2^2\leq	\frac{1}{p} \left\|\A \trans \A\right\|_2^2
	\end{align*}
	and 
	\begin{align*}
		\frac{1}{p} \left\|	\bSig_2-\frac{1}{3}\mathbf{I}_p\right\|_2^2 \leq 	\frac{1}{p} \left\|\frac{2}{\pi} \arcsin(\tilde{\A} \trans \tilde{\A}/2)\right\|_2^2 \leq \frac{1}{4p} \left\|\A \trans \A\right\|_2^2.
	\end{align*}
	Thus, when
	\begin{align*}
		\frac{1}{p} \left\|\A \trans \A\right\|_2^2 \to 0, 
	\end{align*}
	the LSD is still an affine transformation of Mar\u{c}enko-Pastur law \citep{bandeira2017marvcenko}.  If the term $\|\A \trans \A\|_2^2/p$ is large, the result violates the affine transformation of Mar\u{c}enko-Pastur law.  To demonstrate these results, we consider two covariance matrices:
	\begin{align*}
		\bSig_0=\mathbf{I}_p+\frac{1}{p}\Z \trans \Z,~\mbox{and}~	\bSig_0=\mathbf{I}_p+\frac{1}{\sqrt{p}}\Z \trans \Z,
	\end{align*}
where $\Z=(Z_{ij})_{k\times p}$ and $Z_{ij}~i.i.d. \sim N(0,1)$.  Figure~\ref{fig:factor} shows the results which are consistent with our analysis.
	\begin{figure}[!h]
\includegraphics[width=\linewidth]{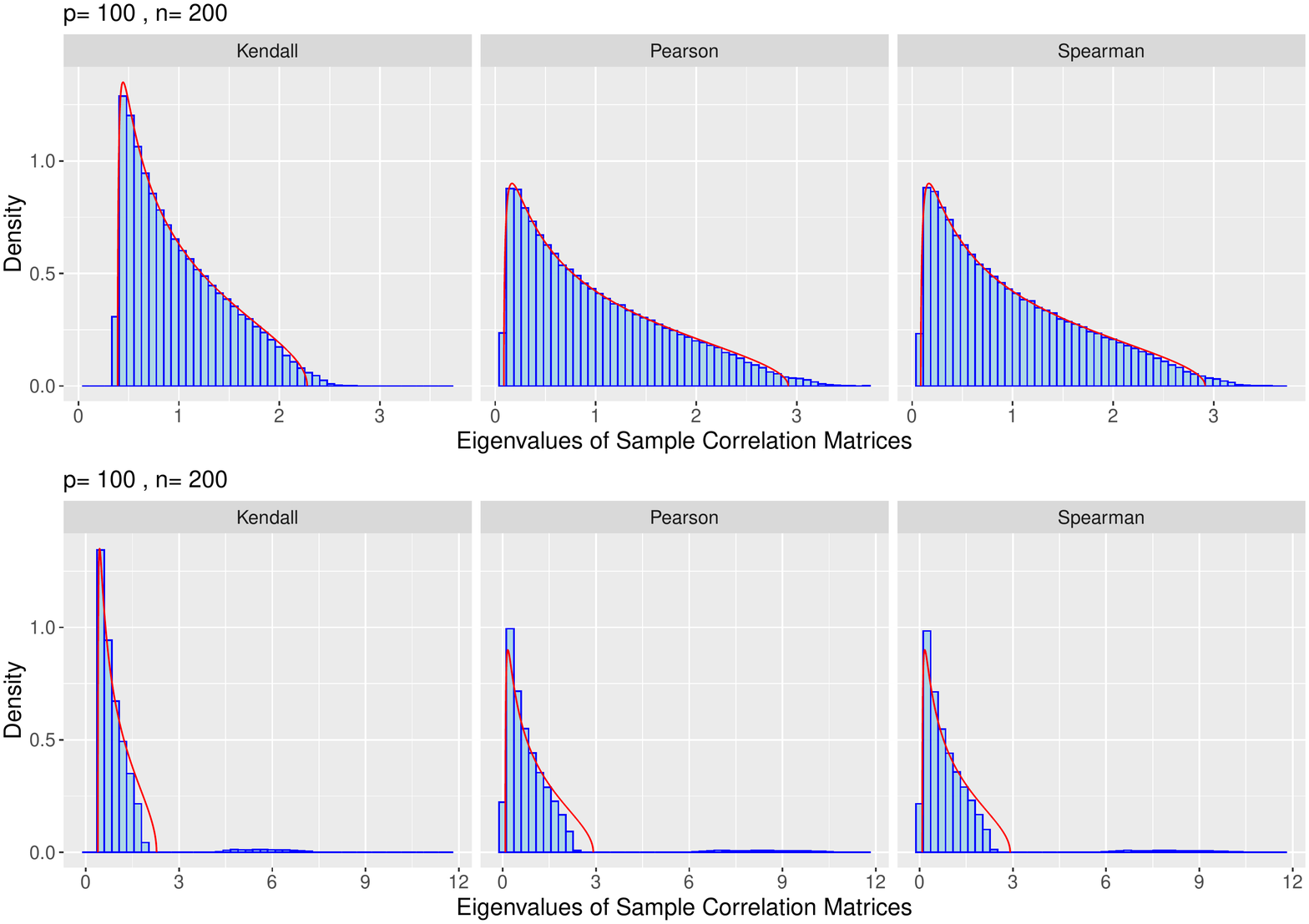}
		\caption{Histograms of eigenvalues of three types of sample correlation matrices, Kendall $\K_n$, Pearson $\bR_n$, Spearman $\bS_n$ for data sample $\bx_1,\cdots,\bx_n\sim N_p(0,\bI_p+\Z \trans \Z/p)$ in the upper panel and $\bx_1,\cdots,\bx_n\sim N_p(0,\bI_p+\Z \trans \Z/\sqrt{p})$ in the lower panel, $\Z=(Z_{ij})_{k\times p},~ Z_{ij}\sim N(0,1)$ i.i.d. with $k=3$, $(p,n)=(100, 200)$ and $(p,n)=(200, 100)$. The red curves are density functions of the affine transformation of Mar\u{c}enko-Pastur law.}
		\label{fig:factor}
	\end{figure}

\vspace{0.2cm}
\noindent\textbf{Model (III): }\\
As for MA(1) model, we focus on the spectral behavior of $\K_n$ since little is known about $\bS_n$ and $\bR_n$ in the dependent case. Similarly, we generate data sample $\bx_1,\ldots,\bx_n\sim N(0,\bSig)$ where $\bSig$ follows MA(1) model (III) with $\rho=0.5$. The LSD are derived using Proposition~\ref{case2} and the inversion formula~\eqref{eq:inv}. Then the histogram of eigenvalues of $\K_n$ under different combinations of $(p,n)$ are compared with their corresponding limiting densities in Figure~ \ref{fig:ma}. It can be seen from Figure~\ref{fig:ma} that the LSDs under MA(1) model are different from the independent case. All empirical histograms conform to our theoretical limits, which proves the accuracy of our theory.
\begin{figure}[!h]
\includegraphics[width=\linewidth]{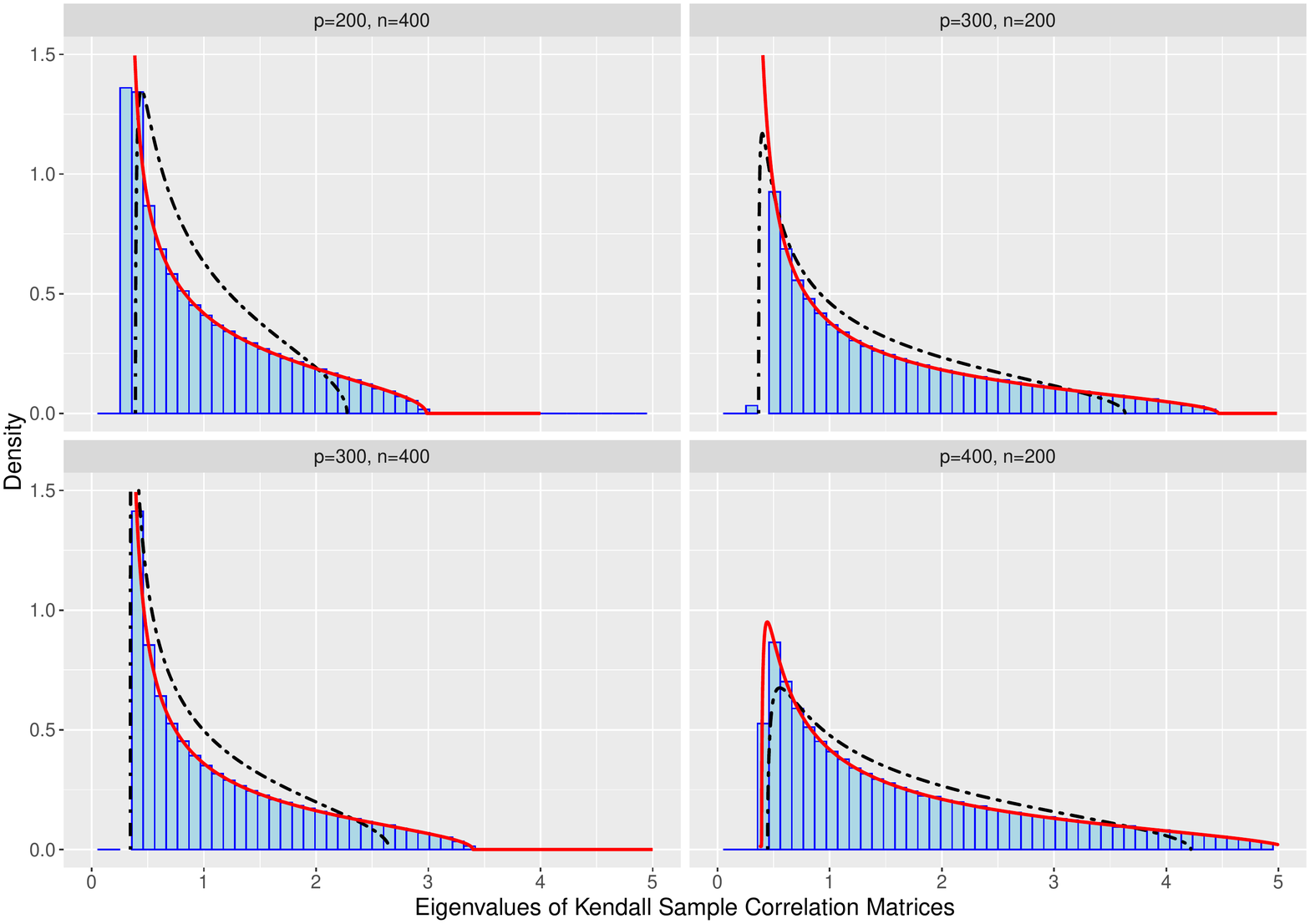}
	\caption{Histograms of eigenvalues of Kendall sample correlation matrices for data sample $\bx_1,\cdots,\bx_n\sim N_p(0,\bSig)$ where $\bSig$ follows MA(1) Model (II) with $\rho=0.5$ for $(p,n)=(200,400)$, $(p,n)=(300,400)$, $(p,n)=(300,200)$ and $(p,n)=(400,200)$. The red curves are density functions of their corresponding limiting spectral distribution. The black dashed line are densities in the independent case for reference.}
	\label{fig:ma}
\end{figure}

\vspace{0.2cm}
\noindent\textbf{Model (IV): }\\
As an illustration for the general Toeplitz matrix, we consider a band Toeplitz matrix with two parameters, i.e.,
		\begin{align*}
			\rho_1=\rho_2=\rho,~\rho_k=0, k=3,\cdots.
		\end{align*}
		Noting
		\begin{align*}
			\bSig_2=\frac{2}{\pi} \arcsin(\bSig/2), ~\bSig_3=& \frac{2}{\pi} \arcsin(\bSig)-\frac{4}{\pi} \arcsin(\bSig/2),
		\end{align*}	
		we have
		\begin{align*}
			\bSig_3-\frac{1}{3} \bI_p=a(	\bSig_2-\frac{1}{3} \bI_p),~a=\frac{\arcsin(\rho)}{ \arcsin(\rho/2)}-2
		\end{align*}
		and then 
		\begin{align*}
			\bSig_3=a \bSig_2+\frac{1-a}{3} \bI_p.
		\end{align*}
		The assumption \eqref{ass2} is
		\begin{align*}
			\lim_{p \to \infty}\frac{1}{p}\tr \big[((2x+a) \bSig_2-(z-\frac{1-a}{3})\bI_p)^{-1}\bSig_2 \big]=\frac{1-x}{2cx}.
		\end{align*}
and Proposition \ref{case3} yields
\begin{align}\label{eqr1}
	\frac{1-x}{2cx}=\frac{1}{2\pi} \int_0^{2\pi} \frac{f_2(\theta)}{(2x+a)f_2(\theta)-(z-\frac{1-a}{3})} d\theta
\end{align}
where 
	\begin{align*}
			f_2(\theta)=&\frac{1}{3}+\frac{4}{\pi}  \arcsin\frac{\rho}{2} \big[\cos{(\theta)}+\cos{(2\theta)}\big].	
	\end{align*}
Solving \eqref{eqr1} to get $x(z)$, the Stieltjes transform of the LSD \eqref{lsd} is
		\begin{align*}
			s(z)=&\lim_{p \to \infty}\frac{1}{p}\tr \left((2x+a) \bSig_2-(z-\frac{1-a}{3})\bI_p\right)^{-1}\\
			=&-\frac{1}{z_1}\lim_{p \to \infty}\frac{1}{p}\tr \left((2x+a) \bSig_2-z_1\bI_p\right)^{-1} \left((2x+a) \bSig_2-z_1\bI_p-(2x+a)\bSig_2 \right) \\
			=&-\frac{1}{z_1} \left(1-\frac{(2x+a)(1-x)}{2cx}\right)=\frac{(2x+a)(1-x)-2cx}{2cz_1x},
		\end{align*}
where 
\begin{align*}
	z_1=z-\frac{1-a}{3}. 
\end{align*}
Figure~\ref{fig:gen} shows the results with $\rho=0.25$ and again, we can see  that  the empirical histogram conforms to our theoretical result.
\begin{figure}[!h]
\includegraphics[width=\linewidth]{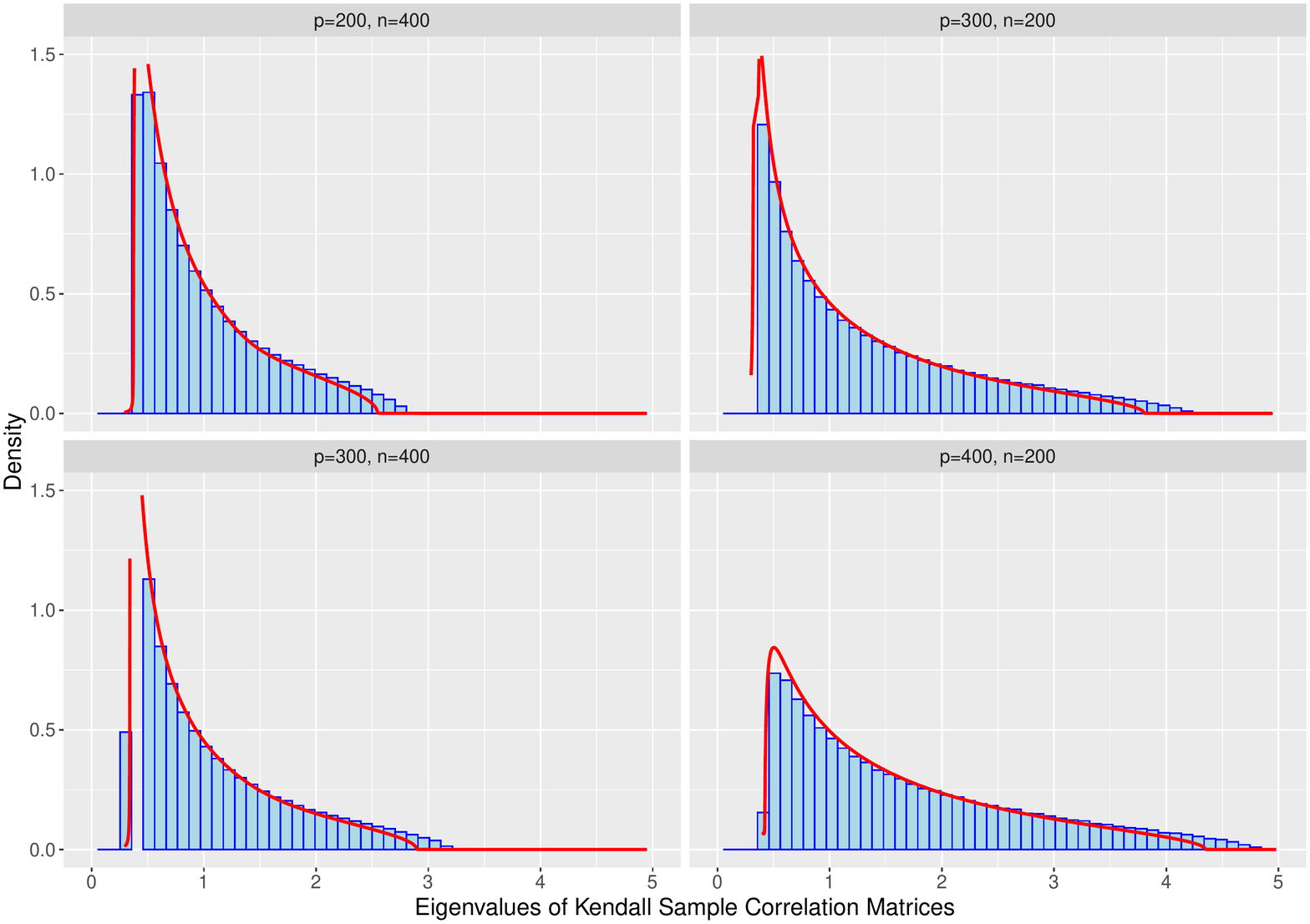}
	\caption{Histograms of eigenvalues of Kendall sample correlation matrices for data sample $\bx_1,\cdots,\bx_n\sim N_p(0,\bSig)$ where $\bSig$ follows general Toeplitz Model (IV) with $\rho=0.25$ for $(p,n)=(200,400)$, $(p,n)=(300,400)$, $(p,n)=(300,200)$ and $(p,n)=(400,200)$. The red curves are density functions of their corresponding limiting spectral distribution.}
	\label{fig:gen}
\end{figure}

\section*{Acknowledgments}
We thank the Editor, an Associate Editor, and anonymous reviewers for their insightful comments. Zeng Li’s research is partially supported by National Natural Science Foundation of China (NSFC) (No. 12031005 and No. 12101292). Cheng Wang's research is supported by NSFC (No. 12031005) and NSF of Shanghai (21ZR1432900). Qinwen Wang's research is partially supported by the NSFC (No. 12171099).

\section*{Appendix}

\setcounter{section}{0}
\setcounter{equation}{0}
\def\theequation{A.\arabic{equation}}
\def\thesection{A\arabic{section}}
This Appendix contains all supporting lemmas and proofs.

\section{Proof of Proposition~\ref{thm1}}
The following results show that $\M_2$ and $\M_3$ are concentrated on their population means, respectively.
	\begin{lemma} \label{lem:error1}
	Under the assumption of Proposition \ref{thm1}, 
	\begin{align*}
		\frac{1}{p}\E\|\M_2\|_2^2\leq  \frac{4 p^2}{3np(n-1)}+\frac{{8}}{np} \tr(\bSig^2_2) \to 0.
	\end{align*}	
\end{lemma}

\begin{lemma} \label{lem:error2}
	Under the assumption of Proposition\ref{thm1}, 
	\begin{align*}
		\frac{1}{p}	\E\|\M_3-\bSig_3\|_2^2\leq\frac{2 p^2}{3np(n-1)}+\frac{32}{np} \tr\{\bSig_1(\bSig_1+\bSig_2)\} \to 0.
	\end{align*}	
\end{lemma}
Equipped with these two results, we are now ready to prove Proposition \ref{thm1}. By the Corollary A.41 of \cite{bai2010spectral}, 
		\begin{align*}
			L^3(F^{\K_n}, F^{\W_n}) \leq \frac{1}{p} \|\K_n-\W_n\|_2^2 \leq \frac{6}{p} \|\M_2\|_2^2+\frac{3}{p} \|\M_3-\bSig_3\|_2^2
		\end{align*}
which yields 
\begin{align*}
\E 	L^3(F^{\K_n}, F^{\W_n})  \to 0.
\end{align*}
The proof  is completed.  \hfill		
\bigskip

It remains to prove the two auxiliary Lemmas \ref{lem:error1} and \ref{lem:error2}.
	
\subsection{Proof of Lemma~\ref{lem:error1}}
Writing the kernel function
	\begin{align*}
		\h(i,j)=(\A_i-\A_j)(\A_{ij}-\A_i+\A_j)\trans,
	\end{align*}
we have 
	\begin{align*}
		\M_2&=\frac{2}{n(n-1)} \sum_{1\leq i<j\leq n}(\A_i-\A_j)(\A_{ij}-\A_i+\A_j)\trans\\
		&=\frac{2}{n(n-1)} \sum_{1\leq i<j\leq n}\h(i,j).	
	\end{align*}
	For the kernel function, we have the following properties.
	\begin{itemize}
		\item For the mean parts, 
		\begin{align*}
			\h(i,j)=\h(j,i),\quad \E \h(1,2)=\mathbf{0},\quad \E \left[\h(1,2)|\bx_1\right]=-\E \left[\A_2 \A_{12}\trans |\bx_1\right]-\bSig_2.
		\end{align*}
		\item For the Frobenius norm of the kernel function, we have
		\begin{align*}
			\E\tr(\h(1,2) \trans \h(1,2))=\E \left(\|\A_{12}-\A_{1}+\A_{2}\|_2^2  \|\A_1-\A_2\|_2^2 \right).
		\end{align*}
		Since
		\begin{align*}
			\|\A_{12}-\A_{1}+\A_{2} \|_\infty=&\|\sign(\A_1-\A_2)-(\A_1-\A_2)\|_\infty 	\leq 1,
		\end{align*}
and $\|\A_i\|_\infty \leq 1$, we have
		\begin{align*}
\|\A_{12}-\A_{1}+\A_{2}\|_2^2\leq p, \mbox{and} \quad  \|\A_1-\A_2\|_2^2 \leq 4p,
		\end{align*}
	which yields 
	\begin{align*}
			\E\tr(\h(1,2) \trans \h(1,2)) \leq 4 p^2.
	\end{align*}
		\item For the the Frobenius norm of the conditional mean, we have
		\begin{align*}
			\E\tr(\h(1,2) \trans \h(1,3))=& \tr\{ \E (\A_2 \A_{12}\trans+\bSig_2)\trans (\A_3 \A_{13}\trans+\bSig_2)  \}\\
			=&	\cov(\A_{13}\trans \A_{12}, \A_2 \trans \A_3)- \tr(\bSig^2_2)\\
	\leq & \E |\A_{13}\trans \A_{12}| |\A_2 \trans \A_3|\leq p \cdot \{\var(\A_{3}\trans \A_{2}) \}^{1/2}\\
=& p \cdot \{\ \tr(\bSig^2_2) \}^{1/2} \leq p^{3/2} \cdot \|\bSig_2\| \leq p^{3/2} \cdot \|\bSig_1\|. 
\end{align*}	
	which yields 
\begin{align*}
	\E\tr(\h(1,2) \trans \h(1,3)) \leq C  p^{3/2}.
\end{align*}
	\end{itemize}
	Putting together the pieces, we conclude that
	\begin{align*}
	\frac{1}{p}	\E \left( \M_2 \trans  \M_2 \right)=	\frac{1}{p}	 &\E \bigg\{\frac{2}{n(n-1)} \sum_{1\leq i<j\leq n} \h(i,j) \bigg\}\trans \bigg\{\frac{2}{n(n-1)} \sum_{1\leq k<l\leq n} \h(k,l) \bigg\}\\
		=&\frac{2}{np(n-1)}	\E\tr(\h(1,2) \trans \h(1,2))+\frac{4p(n-2)}{n(n-1)} 	\E\tr(\h(1,2) \trans \h(1,3))\\
		\leq & \frac{8 p^2}{np(n-1)}+\frac{4C p^{3/2}}{np}\to 0.
	\end{align*}	
	The proof  is completed.  \hfill		
\bigskip

\subsection{Proof of Lemma~\ref{lem:error2}}
	Recalling 
	\begin{align*}
		\M_3=\frac{2}{n(n-1)} \sum_{1\leq i<j\leq n}\epsilon_{ij} \epsilon_{ij}\trans,
	\end{align*}
	we have
	\begin{align*}
		&\E \tr\{(\M_3-\bSig_3)(\M_3-\bSig_3)\trans\}\\
		=&~ \frac{4}{n^2(n-1)^2} \sum_{i<j,k<l}	\E\{\epsilon_{kl} \trans (\epsilon_{ij} \epsilon_{ij}\trans-\bSig_3) 	\epsilon_{kl}\}\\
		=&~\frac{2}{n(n-1)} \E\{\epsilon_{12} \trans (\epsilon_{12} \epsilon_{12}\trans-\bSig_3) 	\epsilon_{12}\}+\frac{4(n-2)}{{n(n-1)}} \E\{\epsilon_{12} \trans (\epsilon_{13} \epsilon_{13}\trans-\bSig_3) \epsilon_{12}\}\\
		=&~\frac{2}{n(n-1)} \{\E (\epsilon_{12} \trans \epsilon_{12})^2-\tr(\bSig^2_3)\} +\frac{4(n-2)}{{n(n-1)}} \{\E (\epsilon_{12} \trans \epsilon_{13})^2-\tr(\bSig^2_3)\}\\
		\leq &~ \frac{2}{n(n-1)} \E (\epsilon_{12} \trans \epsilon_{12})^2 +\frac{4}{{n}} \E (\epsilon_{12} \trans \epsilon_{13})^2.
	\end{align*}
	For the first term, we have
	\begin{align*}
		\epsilon_{12} \trans \epsilon_{12}=	\|\A_{12}-\A_{1}+\A_{2}\|_2^2\leq  p  \|\A_{12}-\A_{1}+\A_{2} \|_\infty^2\leq p,
	\end{align*}
	which yields $\E (\epsilon_{12} \trans \epsilon_{12})^2\leq p^2$.
	For the second term, we have
	\begin{align*}
& (\epsilon_{12} \trans \epsilon_{13})^2
		=\{(\A_{12}-\A_{1}+\A_{2})\trans(\A_{13}-\A_{1}+\A_{3})  \}^2\\
		=&~\{(\A_{12}-\A_1)\trans(\A_{13}-\A_1)+\A_2\trans(\A_{13}-\A_1)+\A_3\trans (\A_{12}-\A_1)+\A_2\trans \A_3\}^2\\
		\leq&~4\{(\A_{12}-\A_1)\trans(\A_{13}-\A_1)\}^2+4\{\A_2\trans(\A_{13}-\A_1)\}^2+4\{\A_3\trans (\A_{12}-\A_1)\}^2+4(\A_2\trans \A_3)^2,
	\end{align*}
	and 
	\begin{align*}
\frac{1}{4}	\E (\epsilon_{12} \trans \epsilon_{13})^2 \leq & \E  \{(\A_{12}-\A_1)\trans(\A_{13}-\A_1)\}^2+\E \{\A_2\trans(\A_{13}-\A_1)\}^2\\
&+\E \{\A_3\trans (\A_{12}-\A_1)\}^2+\E (\A_2\trans \A_3)^2\\
=& \E  \{(\A_{12}-\A_1)\trans \A_{13}\}^2- \E  \{(\A_{12}-\A_1)\trans \A_{1}\}^2+2 \tr \{\bSig_2(\bSig_1-\bSig_2) \}+\tr\{\bSig_2^2 \}\\
\leq & \E  \{(\A_{12}-\A_1)\trans \A_{13}\}^2+ 2 \tr\{\bSig_1 \bSig_2 \}=\var(\A_{12}\trans \A_{13})-\var(\A_{12}\trans \A_{1})+2 \tr\{\bSig_1 \bSig_2 \}\\
\leq&  \var(\A_{12}\trans \A_{13})+2 p C^2.
	\end{align*}
Finally, we have
	\begin{align*}
	\frac{1}{p}	\E \tr\{(\M_3-\bSig_3)(\M_3-\bSig_3)\trans\}\leq \frac{2 p^2}{np(n-1)}+\frac{16}{np} \var(\A_{12}\trans \A_{13})+\frac{32}{np} p C^2 \to 0.
	\end{align*}
	The proof  is completed.  \hfill		

\section{Proof of Theorem \ref{thm:thlsd} }\label{sec:mainproof}
According to Proposition~\ref{thm1}, it suffices to study the LSD of the following matrix $\M_n$.
\begin{align}
	\M_n=\frac{2}{n}\sum_{i=1}^n \A_i \A_i \trans+\bSig_3.
\end{align}
Let $s_{n}(z)$ be the Stieltjes transform of $F^{\M_n}$, then the convergence of $F^{\M_n}$ can be determined in three steps:
\begin{itemize}
	\item[] {\em Step 1}: For any fixed $z\in \mathbb C^+$, $s_{n}(z)-\E s_{n}(z)\to0$, almost surely.
	\item[] {\em Step 2}: For any fixed $z\in \mathbb C^+$, $\E s_{n}(z)\to s(z)$ with $s(z)$ satisfies the equations in \eqref{lsd}.
	
	\item[] {\em Step 3}: The uniqueness of the solution $s(z)$ to \eqref{lsd} on the set $\mathbb C^+$.
\end{itemize}
\subsection{
	Almost sure convergence of $s_{n}(z)-\E s_{n}(z)$}
\medskip

Denote 
\begin{align*}
	\M_{n,k}=\frac{2}{n}\sum_{i\neq k}^n \A_i \A_i \trans+\bSig_3.
\end{align*}
Let $\E_0(\cdot)$ be expectation and $\E_k(\cdot)$ be conditional expectation given $\A_{1},\ldots, \A_k$. From the martingale decomposition and the identity
\begin{align}\label{equ0}
	(\M_n-z\bI_p)^{-1}\A_k=\frac{(\M_{n,k}-z\bI_p)^{-1}\A_k}{1+ 2 n^{-1} \A_k\trans(\M_{n, k}-z\bI_p)^{-1}\A_k},	
\end{align}
we have
\begin{align*}
	s_{n}(z)-\E s_{n}(z)=&\frac{1}{p}\sum_{k=1}^n(\E_k-\E_{k-1})\left[\tr(\M_n-z\bI_p)^{-1}-\tr(\M_{n, k}-z\bI_p)^{-1}\right]\\
	=&-\frac{1}{p}\sum_{k=1}^n(\E_k-\E_{k-1})\frac{2n^{-1}  \A_k\trans(\M_{n, k}-z\bI_p)^{-2}\A_k}{1+ 2n^{-1} \A_k\trans(\M_{n, k}-z\bI_p)^{-1}\A_k}\\
	\defby &-\frac{1}{p}\sum_{k=1}^n(\E_k-\E_{k-1})r_k.
\end{align*}
Since $$|r_k|\leq \frac{\left|2n^{-1}  \A_k\trans(\M_{n, k}-z\bI_p)^{-2}\A_k\right|}{\left|\Im(1+ 2n^{-1} \A_k\trans(\M_{n, k}-z\bI_p)^{-1}\A_k)\right|}\leq \frac{1}{\Im(z)},$$
$\{(\E_k-\E_{k-1})r_k\}$ forms a bounded martingale difference sequence.
Hence for any $\ell>1$,
\begin{align}\label{sne}
	\E|s_{n}(z)-\E s_{n}(z)|^\ell&\leq Kp^{-\ell}\E\bigg(\sum_{k=1}^n\big|(\E_k-\E_{k-1})r_k\big|^2\bigg)^{\ell/2}\nonumber\\
	&\leq Kp^{-\frac{\ell}{2}}\Im(z)^{-\ell},
\end{align}
which implies  $s_{n}(z)-\E s_{n}(z)\to0$, almost surely.

\medskip

\subsection{Convergence of $\E s_{n}(z)$}

Denote 
\begin{align}\label{xn}
	x_n=x_n(z)=\frac 1n \sum_{k=1}^n \frac{1}{1+2n^{-1}\E \left[\tr \left\{(\M_{n,k}-z\bI_p)^{-1}\bSig_2\right\}\right]}.
\end{align}

Starting from the identity 
\begin{align*}
	&\quad \left(2x_n\bSig_2+\bSig_3-z\bI_p\right)^{-1}-(\M_n-z\bI_p)^{-1}\\
	&=\left(2x_n\bSig_2+\bSig_3-z\bI_p\right)^{-1}\bigg(\frac{2}{n}\sum_{k=1}^n \A_k \A_k -2x_n\bSig_2\bigg)(\M_n-z\bI_p)^{-1},
\end{align*}
then take trace on both sides and by \eqref{equ0}, we have 
\begin{align}\label{dof}
	&~\quad\frac 1p \tr(2x_n\bSig_2+\bSig_3-z\bI_p)^{-1}-s_n(z)\nonumber\\
	&=\frac{2}{np}\sum_{k=1}^n \frac{\A_k\trans(\M_{n, k}-z\bI_p)^{-1}(2x_n\bSig_2+\bSig_3-z\bI_p)^{-1}\A_k}{1+ 2n^{-1} \A_k\trans(\M_{n, k}-z\bI_p)^{-1}\A_k}\\
	&\quad-\frac{2x_n}{p}\tr \left\{(2x_n\bSig_2+\bSig_3-z\bI_p)^{-1}\bSig_2(\M_n-z\bI_p)^{-1}\right\}\nonumber\\
	&\defby\frac 2n \sum_{k=1}^n \frac{d_k}{1+2n^{-1}\E \left[\tr \left\{(\M_{n,k}-z\bI_p)^{-1}\bSig_2\right\}\right]},
\end{align}
where
\begin{align*}
	d_k&=\frac{1+2n^{-1}\E \left[\tr \left\{(\M_{n,k}-z\bI_p)^{-1}\bSig_2\right\}\right]}{1+2n^{-1} \A_k\trans(\M_{n,k}-z\bI_p)^{-1}\A_k}\cdot \frac 1p \A_k\trans(\M_{n,k}-z\bI_p)^{-1}(2x_n\bSig_2+\bSig_3-z\bI_p)^{-1}\A_k\\
	&\quad \quad-\frac 1p  \tr \left\{(\M_{n}-z\bI_p)^{-1}(2x_n\bSig_2+\bSig_3-z\bI_p)^{-1}\bSig_2\right\}.
\end{align*}
We decompose $$d_k\defby d_{k1}+d_{k2}+d_{k3},$$
where 
\begin{align*}
	d_{k1}&=\frac 1p  \tr \left\{(\M_{n, k}-z\bI_p)^{-1}(2x_n\bSig_2+\bSig_3-z\bI_p)^{-1}\bSig_2\right\}\\
	&\quad-\frac 1p  \tr \left\{(\M_{n}-z\bI_p)^{-1}(2x_n\bSig_2+\bSig_3-z\bI_p)^{-1}\bSig_2\right\},\\
	d_{k2}&=\frac 1p  \A_k\trans(\M_{n,k}-z\bI_p)^{-1}(2x_n\bSig_2+\bSig_3-z\bI_p)^{-1}\A_k\\
	&\quad -\frac 1p  \tr \left\{(\M_{n, k}-z\bI_p)^{-1}(2x_n\bSig_2+\bSig_3-z\bI_p)^{-1}\bSig_2\right\},\\
	d_{k3}&=\frac{-2 \A_k\trans(\M_{n,k}-z\bI_p)^{-1}(2x_n\bSig_2+\bSig_3-z\bI_p)^{-1}\A_k}{pn\left\{1+2n^{-1} \A_k\trans(\M_{n,k}-z\bI_p)^{-1}\A_k\right\}}\\
	&\quad \times \bigg\{\A_k\trans(\M_{n,k}-z\bI_p)^{-1}\A_k-\E \big[ \tr\left\{(\M_{n,k}-z\bI_p)^{-1}\bSig_2\right\}\big]\bigg\}.
\end{align*}

For the term $d_{k1}$, we have
\begin{align}\label{dk1}
	|d_{k1}|&=\left|\frac {2}{pn}  \tr \left\{(\M_{n}-z\bI_p)^{-1} \A_k \A_k\trans (\M_{n, k}-z\bI_p)^{-1}   (2x_n\bSig_2+\bSig_3-z\bI_p)^{-1}\bSig_2\right\}\right|\nonumber\\
	&=\left|\frac{2n^{-1}p^{-1}\A_k\trans (\M_{n, k}-z\bI_p)^{-1}   (2x_n\bSig_2+\bSig_3-z\bI_p)^{-1}\bSig_2 (\M_{n, k}-z\bI_p)^{-1} \A_k}{1+2n^{-1}\A_k\trans (\M_{n, k}-z\bI_p)^{-1}\A_k}\right|\nonumber\\
	&\leq Kp^{-1}\Im(z)^{-2},
\end{align}
then its contribution to \eqref{dof} can be bounded as
\begin{align*}
	\left|\frac 2n \sum_{k=1}^n \frac{d_{k1}}{1+2n^{-1}\E  \big[\tr\left\{(\M_{n,k}-z\bI_p)^{-1}\bSig_2\right\}\big]}\right|
	&\leq  \frac 2n \sum_{k=1}^n \frac{|d_{k1}|}{\Big|1+2n^{-1}\E  \big[\tr\left\{(\M_{n,k}-z\bI_p)^{-1}\bSig_2\right\}\big]\Big|}\\
	&\leq Kp^{-1}\Im(z)^{-2}\to 0.
\end{align*}
For the term $d_{k2}$, we have 
\begin{align}\label{dk2}
	\E( d_{k2})=0.
\end{align}
For the term $d_{k3}$, we have the first part bounded by
\begin{align*}
	\left|\frac{-2 \A_k\trans(\M_{n,k}-z\bI_p)^{-1}(2x_n\bSig_2+\bSig_3-z\bI_p)^{-1}\A_k}{pn\left\{1+2n^{-1} \A_k\trans(\M_{n,k}-z\bI_p)^{-1}\A_k\right\}}\right|\leq Kp^{-1}n^{-1}\Im(z)^{-2}||\A_k||^2.
\end{align*}
So we have 
\begin{align*}
	|\E (d_{k3})|^2\leq K p^{-2}n^{-2}\E ||\A_k||^4\E \Big|\A_k\trans(\M_{n,k}-z\bI_p)^{-1}\A_k-\E \big[\tr\left\{ (\M_{n,k}-z\bI_p)^{-1}\bSig_2\right\}\big]\Big|^2.
\end{align*}
According to Lemma \ref{lem2-var} and  Lemma \ref{bound},
\begin{align*}
	&\E ||\A_k||^4=\var (\A_k\trans \A_k)+\{\E (\A_k\trans \A_k)\}^2\leq \tr (\bSig_1^2)-\tr (\bSig_2^2)+\{\tr (\bSig_2)\}^2=O(p^2),\\
	&\E \Big|\A_k\trans(\M_{n,k}-z\bI_p)^{-1}\A_k-\E\big[ \tr\left\{ (\M_{n,k}-z\bI_p)^{-1}\bSig_2\right\}\big]\Big|^2\\
	&\quad \leq 
	3||\bSig||\cdot \E\big[\tr\left\{ (\M_{n,k}-z\bI_p)^{-2}\bSig_2 \right\}\big]=O(p),
\end{align*}
which  gives 
\begin{align}\label{dk3}
	|\E (d_{k3})|^2\leq K p^{-1}.
\end{align}
Combining \eqref{dof}, \eqref{dk1}, \eqref{dk2} and \eqref{dk3}, we have 
\begin{align}\label{eqt1}
	\E s_n(z)\to \frac 1p \tr(2x\bSig_2+\bSig_3-z\bI_p)^{-1}\defby s(z),
\end{align}
where $x$ is the limit of $x_n$.

Next, we give the equation that $x$ satisfies.
Starting from the quantity $$n^{-1}\E\big[ \tr\left\{ (\M_{n,k}-z\bI_p)^{-1}\bSig_2\right\}\big]$$ at the denominator in \eqref{xn}. From similar arguments as in \eqref{dof}, we can replace the term $(\M_{n,k}-z\bI_p)^{-1}$ by that of $(2x_n\bSig_2+\bSig_3-z\bI_p)^{-1}$, and this leads to the equation
\begin{align}\label{eqt2}
	x=\frac{1}{1+2\lim_{n \to \infty} n^{-1}\tr (2x\bSig_2+\bSig_3-z\bI_p)^{-1}\bSig_2}.
\end{align}
\eqref{eqt1} and \eqref{eqt2} are exactly the  equations  \eqref{lsd} and  \eqref{eq:x} established in  Theorem \ref{thm:thlsd}.

\subsection{The uniqueness of the solution $s(z)$}
We only have to show that the solution $x(z)$ to \eqref{eq:x}, if exists, is unique in $\mathbb{C}^-$.
Now suppose we have two solutions $x_1=x_1(z), x_2=x_2(z) \in \mathbb{C}^-$   to \eqref{eq:x}  for a common $z \in \mathbb{C}^{+}$, then we can obtain 
\begin{align*}
	\frac{1}{x_1}-\frac{1}{x_2}=\lim_{n\to \infty}\frac{2}{n}\tr \big[(\bSig_3+2x_1 \bSig_2-z\bI_p)^{-1}(2x_2\bSig_2-2x_1\bSig_2)(\bSig_3+2x_2 \bSig_2-z\bI_p)^{-1}  \bSig_2\big].
\end{align*}
If $x_1\neq x_2$, then 
\begin{align*}
	1&=\lim_{n\to \infty}\frac{4}{n}\tr \big[x_1\bSig^{1/2}_2(\bSig_3+2x_1 \bSig_2-z\bI_p)^{-1}\bSig^{1/2}_2\cdot x_2\bSig^{1/2}_2(\bSig_3+2x_2 \bSig_2-z\bI_p)^{-1}  \bSig^{1/2}_2\big]\\
	&=\lim_{n\to \infty}\frac{4}{n}\tr \Big[x_1\Big(\bSig^{-1/2}_2(\bSig_3-z\bI_p)\bSig^{-1/2}_2+2x_1\bI_p\Big)^{-1}\cdot x_2\Big(\bSig^{-1/2}_2(\bSig_3-z\bI_p)\bSig^{-1/2}_2+2x_2\bI_p\Big)^{-1}\Big]\\
	&\defby \lim_{n\to \infty}\frac{4}{n}\tr \Big[x_1\big(\Q(z)+2x_1\bI_p\big)^{-1}\cdot x_2\big(\Q(z)+2x_2\bI_p\big)^{-1}\Big],
\end{align*}
where $\Q(z)$ is defined as $$\Q(z)=\bSig^{-1/2}_2(\bSig_3-z\bI_p)\bSig^{-1/2}_2.$$
By the Cauchy-Schwarz inequality, we have
\begin{align}\label{m1}
	1&\leq \bigg\{\lim_{n\to \infty}\frac{4|x_1|^2}{n}\tr\Big[\big(\Q(z)+2x_1\bI_p\big)\big(\Q(\bar{z})+2\bar{x}_1\bI_p\big)\Big]^{-1}\nonumber\\
	&\quad\quad\times  \lim_{n\to \infty} \frac{4|x_2|^2}{n}\tr\Big[\big(\Q(z)+2x_2\bI_p\big)\big(\Q(\bar{z})+2\bar{x}_2\bI_p\big)\Big]^{-1}\bigg\}^{1/2}.
\end{align}
On the other hand,	denote the eigen-decomposition of $\Q(z)$ by
\begin{align*}
	\Q(z)=\sum_{k=1}^p \lambda_k \bv_k \overline{\bv}_k \trans.
\end{align*} 
Then, we have
\begin{align*}
	\lambda_k
	= \overline{\bv}_k \trans \bSig^{-1/2}_2(\bSig_3-z \bI_p) \bSig^{-1/2}_2 \bv_k
	=  \overline{\bv}_k \trans \bSig^{-1/2}_2 \bSig_3 \bSig^{-1/2}_2 \bv_k- z  \overline{\bv}_k \trans \bSig^{-1}_2 \bv_k,
\end{align*}
which yields 
\begin{align*}
	\Im(\lambda_k)=- \Im(z)  \overline{\bv}_k \trans \bSig^{-1}_2 \bv_k<0.
\end{align*}	
Then taking the imaginary part in \eqref{eq:x}, we have
\begin{align*}
	\frac {\Im(\bar x)}{|x|^2}&=\lim_{n\to \infty}\frac{2}{n}\Im\Big(\tr \big(\Q(z)+2x\bI_p\big)^{-1}\Big)\\
	&=\lim_{n\to \infty}\frac{2}{n}\sum_{k=1}^p\frac{\Im(\bar\lambda_k)+2\Im(\bar x)}{|\lambda_k+2x|^2}\\
	&>\lim_{n\to \infty}\frac{4}{n}\sum_{k=1}^p\frac{\Im(\bar x)}{|\lambda_k+2x|^2}.
\end{align*}
Since $\Im(\bar x)>0$, the above inequality yields that		
\begin{align*}
	\frac {1}{|x|^2}>\lim_{n\to \infty}\frac{4}{n}\sum_{k=1}^p\frac{1}{|\lambda_k+2x|^2}=\lim_{n\to \infty}\frac{4}{n}\tr \big(\Q(z)+2x\bI_p\big)^{-1}\big(\Q(\bar z)+2\bar x\bI_p\big)^{-1},
\end{align*}	
which leads to a contradiction with \eqref{m1}. This contradiction proves that $x_1=x_2$ and hence equation  \eqref{eq:x} has at most one solution in $\mathbb{C}^-$.  The proof of this theorem is then complete. \hfill		

\section{Proof of Proposition \ref{thmnorml}}
The proof is checking the assumptions \ref{ass2} for normal distribution which are summarized in Lemmas \ref{lem2-var} and \ref{bound}.  Before proceeding, we need the following variance result for Kendall's correlation.
\begin{lemma}[\citealt{esscher1924method}] \label{clalem2}
	Consider a multivariate normal distribution 
	\begin{align*}
		\begin{pmatrix}
			z_1\\
			z_2\\
			z_3\\
			z_4
		\end{pmatrix} \sim N \left(\begin{pmatrix}
			0\\
			0\\
			0\\
			0
		\end{pmatrix},   \begin{pmatrix}
			1~ & 1/2~&\rho~&\rho/2~\\
			1/2~&1~&\rho/2~&\rho~\\
			\rho~&\rho/2~&1~&1/2~\\
			\rho/2~&\rho~&1/2~&1~
		\end{pmatrix} \right)
	\end{align*}
	where $\rho \in (-1,1)$,  we have 
	\begin{align} \label{case0}
		\E \left\{\prod_{j=1}^4 \sign(z_j)\right\}=\left(\frac{2}{\pi} \arcsin{\rho}\right)^2-\left\{\frac{2}{\pi} \arcsin{(\rho/2)}\right\}^2+\frac{1}{9}.
	\end{align}
\end{lemma}

\begin{lemma}\label{lem2-var}
	Assuming $\bx_1,\cdots, \bx_n, i.i.d. \sim N(\bf{0},\bSig)$, we have 
	\begin{align}
		\var( \A_{12} \trans \A_{13} )=\tr(\bSig^2_1)-\tr(\bSig^2_2).
	\end{align}
\end{lemma}
\begin{proof}[Proof of Lemma~\ref{lem2-var}]
For the covariance part,  we have
\begin{align*}
	\A_{ij}=\sign(\bx_i-\bx_j)\indist \sign(\bx),
\end{align*}
and Grothendieck's Identity shows that for any bi-variate normal vector $(z_1~ z_2)\trans$,
\begin{align*}
	\E \left[\sign(z_1) \sign(z_2)\right]=\frac{2}{\pi} \arcsin\left[\corr(z_1,z_2)\right].
\end{align*}
Thus,
\begin{align*}
	\cov(\A_{ij})=\cov \left[ \sign(\bx)\right]=\frac{2}{\pi} \arcsin\left(\bSig\right)=\bSig_1.
\end{align*}
For $\A_i$,
\begin{align*}
	\cov(\A_{ij}, \A_i)=\E (\A_{ij} \A_{i}\trans)=\E (\A_{i} \A_{i}\trans)=\cov(\A_i),
\end{align*}
and 
\begin{align*}
	\cov(\A_i)=\E\left[\sign(\bx_1-\bx_2) \sign(\bx_1-\bx_3)\trans\right]=\E\left[\sign\left(\frac{\bx_1-\bx_2}{\sqrt{2}}\right) \sign\left(\frac{\bx_1-\bx_3}{\sqrt{2}}\right)\trans\right].
\end{align*}
For the  enlarged random vector, we have
\begin{align*}
	\frac{1}{\sqrt{2}}	\begin{pmatrix}
		\bx_1-\bx_2\\
		\bx_1-\bx_3
	\end{pmatrix} \sim N \left(\bf{0}, \begin{pmatrix}
		\bSig & \bSig/2\\	
		\bSig/2 &  \bSig\\
	\end{pmatrix} \right)
\end{align*}
and thus
\begin{align*}
	\cov(\A_{ij}, \A_i)=\cov(\A_i)=\frac{2}{\pi} \arcsin( \bSig/2).
\end{align*}

Next, we derive the explicit result for $\var( \A_{12} \trans \A_{13} )$. Since
\begin{align*}
	&\var( \A_{12} \trans \A_{13} )=\E ( \A_{12} \trans \A_{13} )^2-\tr^2(\bSig_2)\\
	=&~ \E \bigg\{ \sum_{j=1}^p \sign(x_{1j}-x_{2j})  \sign(x_{1j}-x_{3j})  \bigg\}^2-\tr^2(\bSig_2)\\
	=& ~\E  \bigg\{\sum_{i,j =1}^p \sign(x_{1i}-x_{2i})  \sign(x_{1i}-x_{3i}) \sign(x_{1j}-x_{2j})  \sign(x_{1j}-x_{3j})\bigg\}-\frac{1}{9}p^2,
\end{align*} 
and for any $(i,j)$,
\begin{align*}
	\frac{1}{\sqrt{2}}	\begin{pmatrix}
		x_{1i}-x_{2i}\\
		x_{1i}-x_{3i}\\
		x_{1j}-x_{2j}\\
		x_{1j}-x_{3j}
	\end{pmatrix} \sim N \left(\begin{pmatrix}
		0\\
		0\\
		0\\
		0
	\end{pmatrix},   \begin{pmatrix}
		1 & 1/2&\bSig_{ij}&\bSig_{ij}/2\\
		1/2&1&\bSig_{ij}/2&\bSig_{ij}\\
		\bSig_{ij}&\bSig_{ij}/2&1&1/2\\
		\bSig_{ij}/2&\bSig_{ij}&1/2&1
	\end{pmatrix} \right),
\end{align*}
then by Lemma \ref{clalem2}, 
\begin{align*} 
	\var( \A_{12} \trans \A_{13} )=&\sum_{i,j =1}^p \left[\left\{\frac{2}{\pi} \arcsin(\bSig_{ij})\right\}^2-\left\{\frac{2}{\pi} \arcsin(\bSig_{ij}/2)\right\}^2+\frac{1}{9}\right]-\frac{1}{9}p^2\\
	=&\tr(\bSig^2_1)-\tr(\bSig^2_2).
\end{align*}
The proof is completed.  \hfill		
\end{proof}

\begin{lemma}\label{bound}
	Let $\A= 2 \Phi(\bx)-1$ where $\bx \sim N_p(\bf{0},\bSig)$, then for any non-random $p \times p$ matrix $\B$, we have
	\begin{align*}
		\var(\A \trans \B \A) \leq 3  \|\bSig\| \tr\left(\B \bSig_2 \B\trans\right).
	\end{align*}
\end{lemma}

\begin{proof}[Proof of Lemma~\ref{bound}]
	For $\bx=(x_1,\cdots,x_p) \trans$, we define a function
	\begin{align*}
		g(\bx)&=\left(2 \Phi(\bx)-1 \right) \trans \B\left(2 \Phi(\bx)-1 \right) \\
		&=\sum_{i=1}^p b_{ii} \left(2\Phi(x_i)-1\right)^2+2\sum_{i <j} b_{ij} \left(2\Phi(x_i)-1\right) \left(2\Phi(x_j)-1\right).
	\end{align*} 
	Direct calculations can show that
	\begin{align*}
		\nabla g(\bx)=4 \diag\left(f(x_1),\cdots,f(x_p) \right) \B \Phi(\bx),
	\end{align*}	
	where $f(x)=\exp(-x^2/2)/\sqrt{2\pi}$. 
	When $\bx \sim N(\bf{0},\bSig)$, by Gaussian Poincar\'{e} inequality, we have
	\begin{align*}
		\var\left(g(\bx)\right)=&\var(\A \trans \B \A) \leq \E \tr\left\{ \nabla g(\bx)  \trans \bSig  \nabla g(\bx) \right\}\\
		\leq& \frac{16}{2\pi} \|\bSig\| \cdot\E \left\{\tr(\A\trans \B\trans \B \A)\right\}\leq 3  \|\bSig\|\cdot \tr\left(\B \bSig_2 \B\trans\right).
	\end{align*}
	The proof  is completed.  \hfill	
\end{proof}

\section{Proof of Proposition \ref{case2}}
We consider a more general matrix
\begin{align*}
	\bSig=\begin{pmatrix}
		a & b &  \\
		b & a & b\\
		&\ddots&\ddots&\ddots\\
		&&b&a &b\\
		&&0&b&a
	\end{pmatrix}, 
\end{align*}
whose eigenvalues are 
\begin{align*}
	\lambda_k=a+2 b \cos\frac{k \pi}{p+1},~k=1,\ldots,p,
\end{align*}
and  the corresponding eigenvectors are 
\begin{align*}
	\bu_k=\sqrt{\frac{2}{p+1}}\bigg(\sin\frac{k \pi}{p+1},\sin\frac{2 k \pi}{p+1}, \cdots, \sin\frac{p k \pi}{p+1}\bigg) \trans.
\end{align*} 
Then,
\begin{align*}
	\frac{1}{p}\tr (\bSig)^{-1}=\frac{1}{p} \sum_{k=1}^p \frac{1}{a+2 b \cos\frac{k \pi}{p+1}}	 \to \frac{1}{2\pi} \int_0^{2\pi} \frac{1}{a+2 b \cos \theta} d\theta=\frac{1}{\sqrt{a^2-4b^2}},
\end{align*}	
where the limit is due to Szeg\"{o}  theorem \citep{gray2006toeplitz} and the integral can be calculated through the residue theorem from complex analysis. 

For $\bSig=\bSig(\rho)$, we have
\begin{align*}
	\bSig_2^{i,i}=\frac{1}{3},~	\bSig_2^{i,i+1}=\bSig_2^{i-1,i}=\frac{2}{\pi} \arcsin\frac{\rho}{2}
\end{align*}
and 
\begin{align*}
	\bSig_3^{i,i}=\frac{1}{3},~	\bSig_3^{i,i+1}=\bSig_3^{i-1,i}=\frac{2}{\pi} \arcsin {\rho}-\frac{4}{\pi} \arcsin\frac{\rho}{2}.
\end{align*}
Thus,
\begin{align*}
	\left(\bSig_3+2x \bSig_2-z\bI_p \right)^{i,i}=&\frac{1}{3}+\frac{2x}{3}-z,\\
	\left(\bSig_3+2x \bSig_2-z\bI_p \right)^{i,i+1}&=\left(\bSig_3+2x \bSig_2-z\bI_p \right)^{i-,i}=\frac{2}{\pi} \arcsin {\rho}+\frac{4(x-1)}{\pi} \arcsin\frac{\rho}{2}.
\end{align*}
This yields
\begin{align*}
	\lim_{p \to \infty}\frac{1}{p}\tr (\bSig_3+2x \bSig_2-z\bI_p)^{-1}=\frac{1}{\sqrt{\Big(\frac{1}{3}+\frac{2x}{3}-z\Big)^2-4\Big(\frac{2}{\pi} \arcsin {\rho}+\frac{4(x-1)}{\pi} \arcsin\frac{\rho}{2}\Big)^2}}
\end{align*}
and 
\begin{align*}
	&\frac{1}{p}\tr \big[(\bSig_3+2x \bSig_2-z\bI_p)^{-1}\bSig_2\big]\\
	=&\frac{1}{p} \sum_{k=1}^p \frac{\frac{1}{3}+ \frac{4}{\pi} (\arcsin\frac{\rho}{2}) \cos\frac{k \pi}{p+1} }{\frac{1}{3}+\frac{2x}{3}-z+2 \Big(\frac{2}{\pi} \arcsin {\rho}+\frac{4(x-1)}{\pi} \arcsin\frac{\rho}{2}\Big) \cos\frac{k \pi}{p+1}}\\
	=&c(x,\rho)+\left(\frac{1}{3}-c(x,\rho) \Big(\frac{1}{3}+\frac{2x}{3}-z\Big) \right) \\
	&\quad \times\frac{1}{p} \sum_{k=1}^p \frac{1}{\frac{1}{3}+\frac{2x}{3}-z+2 \Big(\frac{2}{\pi} \arcsin {\rho}+\frac{4(x-1)}{\pi} \arcsin\frac{\rho}{2}\Big) \cos\frac{k \pi}{p+1}}\\
	\to & c(x,\rho)+\frac{1-c(x,\rho)(1+2x-3z)}{3} \frac{1}{\sqrt{\Big(\frac{1}{3}+\frac{2x}{3}-z\Big)^2-4\Big(\frac{2}{\pi} \arcsin {\rho}+\frac{4(x-1)}{\pi} \arcsin\frac{\rho}{2}\Big)^2}}.
\end{align*}
The proof is completed.  \hfill

\bibliographystyle{imsart-nameyear}
\bibliography{ref}

\begin{thebibliography}{35}

\bibitem[\protect\citeauthoryear{Bai and Silverstein}{2004}]{bai2004clt}
\begin{barticle}[author]
\bauthor{\bsnm{Bai},~\bfnm{ZD}\binits{Z.}} \AND
  \bauthor{\bsnm{Silverstein},~\bfnm{Jack~W}\binits{J.~W.}}
(\byear{2004}).
\btitle{{CLT} for linear spectral statistics of large-dimensional sample
  covariance matrices}.
\bjournal{The Annals of Probability}
\bvolume{32}
\bpages{553--605}.
\end{barticle}
\endbibitem

\bibitem[\protect\citeauthoryear{Bai and Silverstein}{2010}]{bai2010spectral}
\begin{bbook}[author]
\bauthor{\bsnm{Bai},~\bfnm{Zhidong}\binits{Z.}} \AND
  \bauthor{\bsnm{Silverstein},~\bfnm{Jack~W}\binits{J.~W.}}
(\byear{2010}).
\btitle{Spectral analysis of large dimensional random matrices}
\bvolume{20}.
\bpublisher{Springer}.
\end{bbook}
\endbibitem

\bibitem[\protect\citeauthoryear{Bai and Yin}{1993}]{bai1993limit}
\begin{barticle}[author]
\bauthor{\bsnm{Bai},~\bfnm{ZD}\binits{Z.}} \AND
  \bauthor{\bsnm{Yin},~\bfnm{YQ}\binits{Y.}}
(\byear{1993}).
\btitle{Limit of the Smallest Eigenvalue of a Large Dimensional Sample
  Covariance Matrix}.
\bjournal{Annals of Probability}
\bvolume{21}
\bpages{1275--1294}.
\end{barticle}
\endbibitem

\bibitem[\protect\citeauthoryear{Bai and Zhou}{2008}]{bai2008large}
\begin{barticle}[author]
\bauthor{\bsnm{Bai},~\bfnm{Zhidong}\binits{Z.}} \AND
  \bauthor{\bsnm{Zhou},~\bfnm{Wang}\binits{W.}}
(\byear{2008}).
\btitle{Large sample covariance matrices without independence structures in
  columns}.
\bjournal{Statistica Sinica}
\bpages{425--442}.
\end{barticle}
\endbibitem

\bibitem[\protect\citeauthoryear{Bai et~al.}{2009}]{bai2009corrections}
\begin{barticle}[author]
\bauthor{\bsnm{Bai},~\bfnm{Zhidong}\binits{Z.}},
  \bauthor{\bsnm{Jiang},~\bfnm{Dandan}\binits{D.}},
  \bauthor{\bsnm{Yao},~\bfnm{Jian-Feng}\binits{J.-F.}} \AND
  \bauthor{\bsnm{Zheng},~\bfnm{Shurong}\binits{S.}}
(\byear{2009}).
\btitle{Corrections to LRT on large-dimensional covariance matrix by RMT}.
\bjournal{The Annals of Statistics}
\bvolume{37}
\bpages{3822--3840}.
\end{barticle}
\endbibitem

\bibitem[\protect\citeauthoryear{Bandeira, Lodhia and
  Rigollet}{2017}]{bandeira2017marvcenko}
\begin{barticle}[author]
\bauthor{\bsnm{Bandeira},~\bfnm{Afonso~S}\binits{A.~S.}},
  \bauthor{\bsnm{Lodhia},~\bfnm{Asad}\binits{A.}} \AND
  \bauthor{\bsnm{Rigollet},~\bfnm{Philippe}\binits{P.}}
(\byear{2017}).
\btitle{Marcenko-{P}astur law for {K}endall’s tau}.
\bjournal{Electronic Communications in Probability}
\bvolume{22}.
\end{barticle}
\endbibitem

\bibitem[\protect\citeauthoryear{Bao}{2019a}]{bao2019tracy_sp}
\begin{barticle}[author]
\bauthor{\bsnm{Bao},~\bfnm{Zhigang}\binits{Z.}}
(\byear{2019}a).
\btitle{Tracy--Widom limit for {S}pearman’s rho}.
\bjournal{Preprint}.
\end{barticle}
\endbibitem

\bibitem[\protect\citeauthoryear{Bao}{2019b}]{bao2019tracy}
\begin{barticle}[author]
\bauthor{\bsnm{Bao},~\bfnm{Zhigang}\binits{Z.}}
(\byear{2019}b).
\btitle{{T}racy--{W}idom limit for {K}endall’s tau}.
\bjournal{Annals of Statistics}
\bvolume{47}
\bpages{3504--3532}.
\end{barticle}
\endbibitem

\bibitem[\protect\citeauthoryear{Bao, Pan and Zhou}{2012}]{bao2012tracy}
\begin{barticle}[author]
\bauthor{\bsnm{Bao},~\bfnm{Zhigang}\binits{Z.}},
  \bauthor{\bsnm{Pan},~\bfnm{Guangming}\binits{G.}} \AND
  \bauthor{\bsnm{Zhou},~\bfnm{Wang}\binits{W.}}
(\byear{2012}).
\btitle{{T}racy-{W}idom law for the extreme eigenvalues of sample correlation
  matrices}.
\bjournal{Electronic Journal of Probability}
\bvolume{17}
\bpages{1--32}.
\end{barticle}
\endbibitem

\bibitem[\protect\citeauthoryear{Bao et~al.}{2015}]{bao2015Spearman}
\begin{barticle}[author]
\bauthor{\bsnm{Bao},~\bfnm{Zhigang}\binits{Z.}},
  \bauthor{\bsnm{Lin},~\bfnm{Liang-Ching}\binits{L.-C.}},
  \bauthor{\bsnm{Pan},~\bfnm{Guangming}\binits{G.}} \AND
  \bauthor{\bsnm{Zhou},~\bfnm{Wang}\binits{W.}}
(\byear{2015}).
\btitle{{Spectral statistics of large dimensional Spearman’s rank correlation
  matrix and its application}}.
\bjournal{Annals of Statistics}
\bvolume{43}
\bpages{2588 -- 2623}.
\end{barticle}
\endbibitem

\bibitem[\protect\citeauthoryear{Childs}{1967}]{childs1967reduction}
\begin{barticle}[author]
\bauthor{\bsnm{Childs},~\bfnm{Donald~R}\binits{D.~R.}}
(\byear{1967}).
\btitle{Reduction of the multivariate normal integral to characteristic form}.
\bjournal{Biometrika}
\bvolume{54}
\bpages{293--300}.
\end{barticle}
\endbibitem

\bibitem[\protect\citeauthoryear{El~Karoui}{2009}]{el2009concentration}
\begin{barticle}[author]
\bauthor{\bsnm{El~Karoui},~\bfnm{Noureddine}\binits{N.}}
(\byear{2009}).
\btitle{Concentration of measure and spectra of random matrices: Applications
  to correlation matrices, elliptical distributions and beyond}.
\bjournal{The Annals of Applied Probability}
\bvolume{19}
\bpages{2362--2405}.
\end{barticle}
\endbibitem

\bibitem[\protect\citeauthoryear{Esscher}{1924}]{esscher1924method}
\begin{barticle}[author]
\bauthor{\bsnm{Esscher},~\bfnm{Fredrick}\binits{F.}}
(\byear{1924}).
\btitle{On a method of determining correlation from the ranks of the variates}.
\bjournal{Scandinavian Actuarial Journal}
\bvolume{1924}
\bpages{201--219}.
\end{barticle}
\endbibitem

\bibitem[\protect\citeauthoryear{F{\'e}ral and
  P{\'e}ch{\'e}}{2009}]{feral2009largest}
\begin{barticle}[author]
\bauthor{\bsnm{F{\'e}ral},~\bfnm{Delphine}\binits{D.}} \AND
  \bauthor{\bsnm{P{\'e}ch{\'e}},~\bfnm{Sandrine}\binits{S.}}
(\byear{2009}).
\btitle{The largest eigenvalues of sample covariance matrices for a spiked
  population: diagonal case}.
\bjournal{Journal of Mathematical Physics}
\bvolume{50}
\bpages{073302}.
\end{barticle}
\endbibitem

\bibitem[\protect\citeauthoryear{Gao et~al.}{2017}]{gao2017high}
\begin{barticle}[author]
\bauthor{\bsnm{Gao},~\bfnm{Jiti}\binits{J.}},
  \bauthor{\bsnm{Han},~\bfnm{Xiao}\binits{X.}},
  \bauthor{\bsnm{Pan},~\bfnm{Guangming}\binits{G.}} \AND
  \bauthor{\bsnm{Yang},~\bfnm{Yanrong}\binits{Y.}}
(\byear{2017}).
\btitle{High dimensional correlation matrices: The central limit theorem and
  its applications}.
\bjournal{Journal of the Royal Statistical Society, Series B}
\bvolume{79}
\bpages{677--693}.
\end{barticle}
\endbibitem

\bibitem[\protect\citeauthoryear{Gray}{2006}]{gray2006toeplitz}
\begin{barticle}[author]
\bauthor{\bsnm{Gray},~\bfnm{Robert~M}\binits{R.~M.}}
(\byear{2006}).
\btitle{Toeplitz and circulant matrices: a review}.
\bjournal{Foundations and Trends in Communications and Information Theory}
\bvolume{2}
\bpages{155--240}.
\end{barticle}
\endbibitem

\bibitem[\protect\citeauthoryear{Heiny and Yao}{2020}]{heiny2020limiting}
\begin{barticle}[author]
\bauthor{\bsnm{Heiny},~\bfnm{Johannes}\binits{J.}} \AND
  \bauthor{\bsnm{Yao},~\bfnm{Jianfeng}\binits{J.}}
(\byear{2020}).
\btitle{Limiting distributions for eigenvalues of sample correlation matrices
  from heavy-tailed populations}.
\bjournal{arXiv preprint arXiv:2003.03857}.
\end{barticle}
\endbibitem

\bibitem[\protect\citeauthoryear{Jiang}{2004}]{jiang2004limiting}
\begin{barticle}[author]
\bauthor{\bsnm{Jiang},~\bfnm{Tiefeng}\binits{T.}}
(\byear{2004}).
\btitle{The limiting distributions of eigenvalues of sample correlation
  matrices}.
\bjournal{Sankhy{\=a}: The Indian Journal of Statistics}
\bvolume{66}
\bpages{35--48}.
\end{barticle}
\endbibitem

\bibitem[\protect\citeauthoryear{Johnstone}{2001}]{johnstone2001distribution}
\begin{barticle}[author]
\bauthor{\bsnm{Johnstone},~\bfnm{Iain~M}\binits{I.~M.}}
(\byear{2001}).
\btitle{On the distribution of the largest eigenvalue in principal components
  analysis}.
\bjournal{Annals of Statistics}
\bvolume{29}
\bpages{295--327}.
\end{barticle}
\endbibitem

\bibitem[\protect\citeauthoryear{Kargin}{2015}]{kargin2015subordination}
\begin{barticle}[author]
\bauthor{\bsnm{Kargin},~\bfnm{Vladislav}\binits{V.}}
(\byear{2015}).
\btitle{Subordination for the sum of two random matrices}.
\bjournal{The Annals of Probability}
\bvolume{43}
\bpages{2119--2150}.
\end{barticle}
\endbibitem

\bibitem[\protect\citeauthoryear{Kendall}{1938}]{kendall1938new}
\begin{barticle}[author]
\bauthor{\bsnm{Kendall},~\bfnm{Maurice~G}\binits{M.~G.}}
(\byear{1938}).
\btitle{A new measure of rank correlation}.
\bjournal{Biometrika}
\bvolume{30}
\bpages{81--93}.
\end{barticle}
\endbibitem

\bibitem[\protect\citeauthoryear{Kendall}{1949}]{kendall1949rank}
\begin{barticle}[author]
\bauthor{\bsnm{Kendall},~\bfnm{Maurice~G}\binits{M.~G.}}
(\byear{1949}).
\btitle{Rank and product-moment correlation}.
\bjournal{Biometrika}
\bpages{177--193}.
\end{barticle}
\endbibitem

\bibitem[\protect\citeauthoryear{Knowles and
  Yin}{2017}]{knowles2017anisotropic}
\begin{barticle}[author]
\bauthor{\bsnm{Knowles},~\bfnm{Antti}\binits{A.}} \AND
  \bauthor{\bsnm{Yin},~\bfnm{Jun}\binits{J.}}
(\byear{2017}).
\btitle{Anisotropic local laws for random matrices}.
\bjournal{Probability Theory and Related Fields}
\bvolume{169}
\bpages{257--352}.
\end{barticle}
\endbibitem

\bibitem[\protect\citeauthoryear{Leung and Drton}{2018}]{leung2018testing}
\begin{barticle}[author]
\bauthor{\bsnm{Leung},~\bfnm{Dennis}\binits{D.}} \AND
  \bauthor{\bsnm{Drton},~\bfnm{Mathias}\binits{M.}}
(\byear{2018}).
\btitle{Testing independence in high dimensions with sums of rank
  correlations}.
\bjournal{Annals of Statistics}
\bvolume{46}
\bpages{280--307}.
\end{barticle}
\endbibitem

\bibitem[\protect\citeauthoryear{Li, Wang and Li}{2021}]{ZWL2021Kendall}
\begin{barticle}[author]
\bauthor{\bsnm{Li},~\bfnm{Zeng}\binits{Z.}},
  \bauthor{\bsnm{Wang},~\bfnm{Qinwen}\binits{Q.}} \AND
  \bauthor{\bsnm{Li},~\bfnm{Runze}\binits{R.}}
(\byear{2021}).
\btitle{{Central limit theorem for linear spectral statistics of large
  dimensional Kendall’s rank correlation matrices and its applications}}.
\bjournal{Annals of Statistics}
\bvolume{49}
\bpages{1569 -- 1593}.
\end{barticle}
\endbibitem

\bibitem[\protect\citeauthoryear{Liu, Lafferty and
  Wasserman}{2009}]{liu2009nonparanormal}
\begin{barticle}[author]
\bauthor{\bsnm{Liu},~\bfnm{Han}\binits{H.}},
  \bauthor{\bsnm{Lafferty},~\bfnm{John}\binits{J.}} \AND
  \bauthor{\bsnm{Wasserman},~\bfnm{Larry}\binits{L.}}
(\byear{2009}).
\btitle{The nonparanormal: Semiparametric estimation of high dimensional
  undirected graphs.}
\bjournal{Journal of Machine Learning Research}
\bvolume{10}.
\end{barticle}
\endbibitem

\bibitem[\protect\citeauthoryear{Mar\u{c}enko and
  Pastur}{1967}]{marvcenko1967distribution}
\begin{barticle}[author]
\bauthor{\bsnm{Mar\u{c}enko},~\bfnm{Vladimir~A}\binits{V.~A.}} \AND
  \bauthor{\bsnm{Pastur},~\bfnm{Leonid~Andreevich}\binits{L.~A.}}
(\byear{1967}).
\btitle{Distribution of eigenvalues for some sets of random matrices}.
\bjournal{Mathematics of the USSR-Sbornik}
\bvolume{1}
\bpages{457}.
\end{barticle}
\endbibitem

\bibitem[\protect\citeauthoryear{Mestre and
  Vallet}{2017}]{mestre2017correlation}
\begin{barticle}[author]
\bauthor{\bsnm{Mestre},~\bfnm{Xavier}\binits{X.}} \AND
  \bauthor{\bsnm{Vallet},~\bfnm{Pascal}\binits{P.}}
(\byear{2017}).
\btitle{Correlation tests and linear spectral statistics of the sample
  correlation matrix}.
\bjournal{IEEE Transactions on Information Theory}
\bvolume{63}
\bpages{4585--4618}.
\end{barticle}
\endbibitem

\bibitem[\protect\citeauthoryear{Morales-Jimenez
  et~al.}{2021}]{morales2021asymptotics}
\begin{barticle}[author]
\bauthor{\bsnm{Morales-Jimenez},~\bfnm{David}\binits{D.}},
  \bauthor{\bsnm{Johnstone},~\bfnm{Iain~M}\binits{I.~M.}},
  \bauthor{\bsnm{McKay},~\bfnm{Matthew~R}\binits{M.~R.}} \AND
  \bauthor{\bsnm{Yang},~\bfnm{Jeha}\binits{J.}}
(\byear{2021}).
\btitle{Asymptotics of eigenstructure of sample correlation matrices for
  high-dimensional spiked models}.
\bjournal{Statistica Sinica}
\bvolume{31}
\bpages{571}.
\end{barticle}
\endbibitem

\bibitem[\protect\citeauthoryear{Pillai and Yin}{2012}]{pillai2012edge}
\begin{barticle}[author]
\bauthor{\bsnm{Pillai},~\bfnm{Natesh~S}\binits{N.~S.}} \AND
  \bauthor{\bsnm{Yin},~\bfnm{Jun}\binits{J.}}
(\byear{2012}).
\btitle{Edge universality of correlation matrices}.
\bjournal{Annals of Statistics}
\bvolume{40}
\bpages{1737--1763}.
\end{barticle}
\endbibitem

\bibitem[\protect\citeauthoryear{Schott}{2005}]{schott2005testing}
\begin{barticle}[author]
\bauthor{\bsnm{Schott},~\bfnm{James~R}\binits{J.~R.}}
(\byear{2005}).
\btitle{Testing for complete independence in high dimensions}.
\bjournal{Biometrika}
\bvolume{92}
\bpages{951--956}.
\end{barticle}
\endbibitem

\bibitem[\protect\citeauthoryear{Silverstein and
  Bai}{1995}]{silverstein1995empirical}
\begin{barticle}[author]
\bauthor{\bsnm{Silverstein},~\bfnm{Jack~W}\binits{J.~W.}} \AND
  \bauthor{\bsnm{Bai},~\bfnm{ZD}\binits{Z.}}
(\byear{1995}).
\btitle{On the empirical distribution of eigenvalues of a class of large
  dimensional random matrices}.
\bjournal{Journal of Multivariate Analysis}
\bvolume{54}
\bpages{175--192}.
\end{barticle}
\endbibitem

\bibitem[\protect\citeauthoryear{Wang, Jin and Miao}{2011}]{wang2011limiting}
\begin{barticle}[author]
\bauthor{\bsnm{Wang},~\bfnm{Cheng}\binits{C.}},
  \bauthor{\bsnm{Jin},~\bfnm{Baisuo}\binits{B.}} \AND
  \bauthor{\bsnm{Miao},~\bfnm{Baiqi}\binits{B.}}
(\byear{2011}).
\btitle{On limiting spectral distribution of large sample covariance matrices
  by VARMA (p, q)}.
\bjournal{Journal of Time Series Analysis}
\bvolume{32}
\bpages{539--546}.
\end{barticle}
\endbibitem

\bibitem[\protect\citeauthoryear{Weihs, Drton and
  Meinshausen}{2018}]{weihs2018symmetric}
\begin{barticle}[author]
\bauthor{\bsnm{Weihs},~\bfnm{Luca}\binits{L.}},
  \bauthor{\bsnm{Drton},~\bfnm{Mathias}\binits{M.}} \AND
  \bauthor{\bsnm{Meinshausen},~\bfnm{Nicolai}\binits{N.}}
(\byear{2018}).
\btitle{Symmetric rank covariances: a generalized framework for nonparametric
  measures of dependence}.
\bjournal{Biometrika}
\bvolume{105}
\bpages{547--562}.
\end{barticle}
\endbibitem

\bibitem[\protect\citeauthoryear{Zheng et~al.}{2019}]{zheng2019test}
\begin{barticle}[author]
\bauthor{\bsnm{Zheng},~\bfnm{Shurong}\binits{S.}},
  \bauthor{\bsnm{Cheng},~\bfnm{Guanghui}\binits{G.}},
  \bauthor{\bsnm{Guo},~\bfnm{Jianhua}\binits{J.}} \AND
  \bauthor{\bsnm{Zhu},~\bfnm{Hongtu}\binits{H.}}
(\byear{2019}).
\btitle{Test for high-dimensional correlation matrices}.
\bjournal{Annals of Statistics}
\bvolume{47}
\bpages{2887--2921}.
\end{barticle}
\endbibitem

\end{thebibliography}

\end{document}